\documentclass[journal]{IEEEtran}
\usepackage{textcomp}
\usepackage{enumitem}
\usepackage{amsthm}
\usepackage{graphicx}
\usepackage{subcaption}
\usepackage{epstopdf}
\usepackage{float}
\usepackage[nocomma]{optidef}
\usepackage{dsfont}
\usepackage{times}
\usepackage{bm}
\usepackage{fancyhdr,graphicx,amsmath,amssymb}
\usepackage[ruled,vlined]{algorithm2e}
\usepackage{hyperref}
\hypersetup{
  colorlinks   = true, 
  urlcolor     = blue, 
  linkcolor    = blue, 
  citecolor    = red   
}
\include{pythonlisting}
\usepackage{algorithmic}
\newcounter{algsubstate}


\def\bt#1{{\color{black}#1}}
\def\btt#1{{\color{black}#1}}
\usepackage{cite}
\usepackage{xcolor}
\usepackage{amsmath}
\usepackage{amssymb}
\usepackage{mathtools}

\makeatletter
\def\BState{\State\hskip-\ALG@thistlm}
\makeatother

\usepackage{stfloats}

\begin{document}
\bstctlcite{IEEEexample:BSTcontrol}

\addtolength{\abovedisplayskip}{-.1cm}
\addtolength{\belowdisplayskip}{-.1cm}

\addtolength{\textfloatsep}{-.5cm}

%
\title{Robust Distributed Optimization \\ With Randomly Corrupted Gradients}

\author{\IEEEauthorblockN{Berkay Turan \quad}
\IEEEauthorblockN{C\'esar A. Uribe \quad}
\IEEEauthorblockN{Hoi-To Wai \quad}
\and
\IEEEauthorblockN{Mahnoosh Alizadeh}%
\thanks{B. Turan and M. Alizadeh are with  Dept. of ECE, UCSB, Santa Barbara, CA, USA. C. A. Uribe is with Dept. of ECE, Rice University, TX, USA. H. T. Wai is with Dept. of SEEM, CUHK, Shatin, Hong Kong. This work is supported by  NSF grant \#1847096. H.T. Wai was partially supported by CUHK Direct Grant \#4055113 and C.A. Uribe was partially supported by ARPA-H Strategic Initiative Seed Fund \#916012, and Sustainable Futures Fund \#919027. E-mails: {\url{bturan@ucsb.edu}, \url{cauribe@rice.edu}, \url{htwai@se.cuhk.edu.hk}, \url{alizadeh@ucsb.edu}}}
}


%


\maketitle


\begin{abstract}
In this paper, we propose a first-order distributed optimization algorithm that is provably robust to Byzantine failures—arbitrary and potentially adversarial behavior, where all the participating agents are prone to failure. We model each agent's state over time as a two-state Markov chain that indicates Byzantine or trustworthy behaviors at different time instants. We set no restrictions on the maximum number of Byzantine agents at any given time. We design our method based on three layers of defense: 1) \bt{temporal robust aggregation}, 2) \bt{spatial robust aggregation}, and 3) gradient normalization. 
We study two settings for stochastic optimization, namely Sample Average Approximation and Stochastic Approximation. \bt{We provide convergence guarantees of our method for strongly convex and smooth non-convex cost functions.}
\end{abstract}

%
\IEEEpeerreviewmaketitle
\newtheorem{proposition}{Proposition}
\newtheorem{definition}{Definition}
\newtheorem{corollary}{Corollary}
\newtheorem{theorem}{Theorem}
\newtheorem{lemma}{Lemma}
\newtheorem{Fact}{Fact}
\newtheorem{remark}{Remark}
\newtheorem{assumption}{Assumption}
\newtheorem*{runningexample}{Running Example}
\makeatletter
\def\blfootnote{\xdef\@thefnmark{}\@footnotetext}
\makeatother
\newcommand{\prm}{\boldsymbol{\theta}}
\newcommand{\prmdl}{\boldsymbol{z}}
\newcommand{\eqdef}{\vcentcolon=}
\newcommand{\beq}{\begin{equation}}
\newcommand{\eeq}{\end{equation}}
\newcommand{\grd}{\nabla}
\newcommand{\Cset}{\mathcal C}
\newcommand{\ie}{i.e., }
\renewcommand{\thefootnote}{\fnsymbol{footnote}}
 \renewcommand{\thefootnote}{\arabic{footnote}}
\section{Introduction}
Convenience for large-scale data processing, privacy preservation, and parallel algorithm execution rendered the design of distributed optimization algorithms an attractive field for scholars \cite{yu2017distributed,yang2019survey,mertikopoulos2017distributed,tekin2015distributed,chouvardas2015greedy,marano2013nearest,swenson2015empirical}. However, the distributed nature of such methods, for example, physically separated servers connected over a network, exposes the system to vulnerabilities not faced by their traditional centralized counterparts \cite{chen2018internet}. The robustness and security of distributed methods need to be taken into account when assessing algorithm performance \cite{yang2019survey}.

In a centralized system, data can be cleaned, faultless computation can be established by reliable hardware, and communication requirements are minimal. On the other hand, typical distributed algorithms assume trustworthy data, faultless computation, and reliable communication. Also, privacy constraints might not allow for data corruption checks, while distributed computing infrastructure increases the likelihood of faulty hardware, e.g., personal devices~\cite{konevcny2016federated}.
Lastly, unreliable communication might occur due to noisy wireless communication, or more importantly, due to man-in-the-middle adversarial attacks.
In man-in-the-middle attacks, an adversary can take over network sub-systems and arbitrarily alter the information stored in and communicated between the machines to prevent convergence to the optimal solution, i.e., Byzantine attacks~\cite{vempaty2013distributed}.




Robust distributed optimization under adversarial manipulation has been studied for various corruption models, see~\cite{kairouz2019advances,yang2020adversary} for comprehensive reviews. For example, gradients communicated over a network are usually modeled as corrupted by: non-malicious noise~\cite{mnih2012learning},
adversarial noise~\cite{tramer2019adversarial},
quantization~\cite{alistarh2016qsgd},
or because the gradients are inexact oracles~\cite{dvurechensky2017gradient}.
Although robust optimization methods with strong theoretical guarantees are well established~\cite{natarajan2013learning,tramer2019adversarial}, a drawback of these approaches is that the corrupt gradients are assumed to be within a bounded neighborhood of the trustworthy ones, i.e., corruption can be modeled as a bounded additive noise to the trustworthy gradients.


On the other hand, an adversarial corruption model, which can be unbounded and arbitrary, has been extensively studied in the distributed learning literature under categories of data poisoning \cite{steinhardt2017certified}
and model update poisoning attacks \cite{wu2020federated,bagdasaryan2020backdoor}.
This line of work models corruption as an arbitrary manipulation on the information sent by the agents  or on the data samples stored at the agents. However, the adversary is often assumed to have limited capability, i.e., the adversary is only able to manipulate a certain fraction of agents or data samples. Although successful defense mechanisms based on robust aggregation methods \cite{chen2017distributed,chen2018draco,cao2020distributed,blanchard2017machine,yin2019defending,yin2018byzantine,yang2019byrdie}
and data sanitation using robust statistics \cite{steinhardt2017certified}
are shown to be robust to these types of manipulation, robust estimation techniques rely on a bounded $\alpha$ fraction of agents/data points being corrupt at all times. Therefore, they are not applicable if there exist iterations with more than $\alpha$ fraction of corrupted agents. For instance, if at any iteration more than half of the agents behave unpredictably and send arbitrarily corrupt information, then the aggregate will be arbitrarily corrupted. In fact, it was recently shown that even more benign-looking manipulations are able to get through these defense mechanisms with corruption rates as low as $.5-1\%$~\cite{wang2020attack}.

In this paper, we study another corruption model where existing defense mechanisms are prone to failure. In particular, we adopt a distributed optimization framework where a group of agents communicates local gradient information to a central machine that aggregates and distributes information back to the agents. By modeling the temporal dynamics of the agents' states (either trustworthy or corrupted/Byzantine) via a two-state Markov chain, we allow \emph{all the agents} to be susceptible to  \emph{arbitrary corruption}.  
\bt{
This type of corruption would occur in practical applications of distributed optimization due to various reasons including but not limited to:
\begin{enumerate}[wide]
		    \item Behavioral (intentional or unintentional) changes of the agents: Due to its privacy-preserving nature, the models established for practical applications such as text completion are trained using user text data without observing it. Besides unintentional mistakes that can be made by a user at random times, users can intentionally behave differently at different time periods. For instance, a multilingual person who works in the United States could be typing in English during work hours and in \btt{another language} after work hours. These periods can also be longer or shorter in duration. If the goal is to train a text completion model for English, then we consider the user as Byzantine when they type in \btt{another language} and trustworthy when they type in English. In this setting, a Markovian Byzantine agent model would be a suitable corruption model.
		    \item Cyber attacks in cyber-physical systems: Among the many types of cyber attacks, Byzantine attacks and man-in-the-middle attacks are important to defend against for distributed optimization algorithms. The attacker is free to hack any user at any time, however, the hack is not necessarily successful all the time, for instance, due to the existence of a firewall. In this case, it would be suitable to model the dynamics of a trustworthy agent's state as a Markov chain with a certain probability of turning into Byzantine, which would capture the aforementioned random characteristics of a Byzantine attack. On the other hand, literature on Byzantine fault detection and man-in-the-middle attack detection establishes that with repeated interactions with the agents, these types of attacks can be detected \cite{ding2018survey,bhushan2017man}. However, there are no certain guarantees on how long a successful detection would take as it would depend on how the attacker behaves. To resemble this randomness in the detection time and success, it would be suitable to model the dynamics of a  Byzantine agent's state as a Markov chain with a certain probability of turning into trustworthy. Given these features of cyber attacks and defense on cyber-physical systems, it would be suitable to \btt{approximate the agent behavior by} a Markovian model as opposed to a static model for distributed optimization applications.
		    \end{enumerate}
}

\bt{A consequence of the Markovian setting is that there could exist iterations at which the majority of the agents send corrupt gradients to the central machine, in which case existing defense mechanisms would fail.} For this setting, we develop a robust distributed optimization algorithm with provable convergence guarantees for a number of function classes.

\noindent\textbf{Contributions:} 
Our main contribution is a distributed stochastic optimization algorithm, named Robust \bt{Aggregating} Normalized Gradient Method (RANGE), that achieves \bt{strong convergence guarantees} while being robust to a newly proposed Markovian gradient corruption model.
\begin{itemize}[leftmargin=5mm, noitemsep, topsep = .5mm]
\item We propose a novel Markovian Byzantine agent model that models dynamically changing sets of Byzantine agents with no assumptions on the maximum fraction of Byzantine agents at a particular iteration.
\item We study two settings for stochastic optimization for RANGE, namely Sample Average Approximation (SAA) and Stochastic Approximation (SA). We prove that for both SAA and SA, when the parameters are tuned appropriately according to the spectral gap of the Markov chain, RANGE converges to a neighborhood of the optimal solution at a linear rate for strongly convex cost functions.
\item We prove that for smooth (possibly non-convex) cost functions, RANGE converges to a neighborhood of a stationary point at a rate of ${\cal O}(1/\sqrt{T})$, where $T$ is the number of iterations.
\item For the SAA setting, we show that RANGE achieves lower error rates   in the Markovian Byzantine agent setup with an expected $\alpha$ fraction of Byzantine agents than state-of-the-art algorithms in the setup with a bounded $\alpha$ fraction of Byzantine agents for all iterations.
     \item We show that RANGE achieves lower statistical error rates in the SA setting than the SAA setting for sufficiently low corruption rates, i.e., the expected fraction of Byzantine agents. We provide an explicit characterization of such bound.
    \item We provide numerical evidence demonstrating the efficacy and robustness of RANGE in the proposed setting.
\end{itemize}

RANGE is designed with three ingredients: (1) \bt{temporal robust aggregation}, (2) \bt{spatial robust aggregation}, and (3) gradient normalization. The \bt{temporal robust aggregation} step estimates the robust mean of each agent's historical gradient data over a finite window to compute a robustified gradient for each agent. Informally, the received gradients over a period of time contain a fraction of trustworthy information that can be extracted by the algorithm to perform faithful computations rather than applying potentially corrupt gradients directly. In case the robustified gradient produced by \bt{temporal robust aggregation} becomes contaminated by corruption, another layer of defense mechanism is implemented via \bt{spatial robust aggregation} of all the agents' robustified gradients. Lastly,  normalization preserves only the directional information and thus prevents large updates  that corrupt gradients might cause in case the \bt{temporal robust aggregation} and \bt{spatial robust aggregation} steps do not sufficiently eliminate corruptions. 

\noindent\textbf{Related work:} Our work has connections to the literature on (i)  normalized gradient method, (ii) gradient clipping, and (iii) delayed gradient descent.
\begin{itemize}[wide, labelindent=0pt,topsep=.5mm]
    \item \emph{Normalized Gradient Method: } Normalized gradient method is a well-studied algorithm for optimization and is supported by theoretical convergence guarantees for convex~\cite{nesterov2004introductory}
    and quasi-convex optimization~\cite{hazan2015beyond}. Using normalized updates is gaining popularity, especially for non-convex optimization~\cite{you2017large},
    since for non-convex objectives, unlike the convex ones, the magnitude of the gradient provides less information about the value of the function, while the direction still indicates the direction of steepest descent. 
    An important benefit of this is the fast evasion of saddle points~\cite{levy2016power}.
    Seeing the need for large batch sizes for variance reduction of stochastic gradients as a drawback of normalized updates, a recent work \cite{cutkosky2020momentum} proves that adding momentum removes the need for large batch sizes on non-convex objectives while matching the best-known convergence rates. In a preliminary conference report~\cite{turan2020robustness}, we investigated the robustness properties of the normalized subgradient method for solving deterministic optimization problems in a centralized fashion. In the current work, we expand \cite{turan2020robustness} into a distributed setup with a stochastic objective function, additionally study non-convex objectives both theoretically and numerically, and employ two additional layers of defense by means of robust mean estimation before applying normalization to improve our algorithm.
    \item \emph{Gradient Clipping: }
    As a similar method to normalization, gradient clipping is a common technique in optimization used for privacy \cite{abadi2016deep}.
    Recent studies demonstrate that gradient clipping can be applied for robustness to model update poisoning attacks \cite{sun2019can} and label noise~\cite{menon2020can}. 
    However, similar to robust distributed optimization literature, due to the limitations on the amount of corruption and adversarial agents, their methods are inapplicable in our setting and can be outperformed, as we will show numerically in Section~\ref{sec:numerical}.
    \item \emph{Delayed Gradient Descent: }\bt{Temporal robust aggregation} step of our method is in principle similar to a delayed gradient descent method \cite{nedic2001distributed}, since \bt{temporal aggregation} is a linear combination of the past gradients. Motivated by applications to distributed optimization over networks, researchers have established convergence guarantees for deterministic \cite{gurbuzbalaban2017convergence}
    and stochastic delayed gradient methods \cite{agarwal2012distributed}.
    Given strong theoretical results, we integrate the delayed gradient method to our algorithm via \bt{temporal robust aggregation} and show that it improves robustness.
\end{itemize}
\noindent
\textbf{Paper Organization:}
The remainder of the paper is organized as follows. In Section~\ref{sec:problem}, we define the problem setting. In Section~\ref{sec:range}, we describe our algorithm called RANGE and discuss how it can solve the proposed problem. In Sections~\ref{sec:SAA} and \ref{sec:SA}, we present the convergence properties of RANGE for the SAA and SA settings, respectively. In Section~\ref{sec:special}, we discuss two special cases of RANGE, one without \bt{temporal robust aggregation} and one with independent random corruption. In Section~\ref{sec:numerical}, we provide numerical results for RANGE.

\noindent \textbf{Notations and conventions: } Unless otherwise specified, $\| \cdot \|$ denotes the standard Euclidean norm. For any $N \in \mathbb N$, $[N]$ denotes the finite set $\{1,...,N\}$. Given a vector $v$, if $\|v\|=0$, then $v/\|v\|=0$. The ${\cal O}(\cdot)$ notation hides constants, logarithmic terms, and only includes the dominant terms. Given a function $f(x,z)$, $\partial_k f(x,z)$ denotes the partial derivative of $f(x,z)$ with respect to $k$'th coordinate of $x$.
\section{Problem Setup}\label{sec:problem}
In this section, we formally set up our problem and introduce key concepts and definitions that will be used in the paper. We are interested in the stochastic optimization problem
\begin{equation}\label{eq:mainproblem}
   x^\star=\underset{x\in{\cal X}}{\arg\min}~ F(x)=\underset{x\in{\cal X}}{\arg\min}~\underset{z\sim {\cal D}}{\mathbb{E}}[f(x,z)],
\end{equation}
where $f(x,z)$ is a cost function of a parameter vector $x\in {\cal X}\subseteq \mathbb{R}^d$ associated with a data point $z\in {\cal Z}$ and the data points are sampled from some unknown distribution ${\cal D}$. To solve \eqref{eq:mainproblem}, we study two settings for stochastic optimization, namely Sample Average Approximation (SAA) \cite{kleywegt2002sample} and Stochastic Approximation (SA) \cite{wasan2004stochastic}, in a distributed setup with one central machine and $N$ agents that compute stochastic gradients at a point $x$
via $\nabla f(x, z)$ based on independent samples $z\sim {\cal D}$. 

In iterative distributed first-order methods, given the parameter vector $x_t$ at iteration $t$, the central machine receives the feedback $\nabla F_{i,t}(x_t)$ from all the agents, aggregates by computing the average, and applies a descent step to get the updated parameter $x_{t+1}$.  Here,
\begin{equation}
    F_{i,t}(x_t)=\frac{1}{b}\sum_{j=1}^b f(x_t,z_{i,t}^{j})
\end{equation}
is the empirical risk function and $\{z_{i,t}^{j}\}_{j\in[b]}$ are the $b$ data points used for gradient computation at agent $i$ and iteration $t$. In SAA, each agent uses a fixed set of data samples to estimate the gradient at all iterations, i.e., $\{z_{i,\tau}^j\}_{j\in[b]}=\{z_{i,\tau'}^{j}\}_{j\in[b]}$ and $F_{i,\tau}(x)=F_{i,\tau'}(x)$ $\forall i\in[N],x\in{\cal X},\tau,\tau'\in{\mathbb N}_0$ \cite{kim2015guide}. In SA, the agents sample $b$ new data points from ${\cal D}$ at each iteration and therefore $F_{i,\tau}(x)$ and $F_{i,\tau'}(x)$ are independent random variables $\forall i\in [N]$, $x\in{\cal X}$, $\tau,\tau'\in{\mathbb N}_0$ such that $\tau\neq \tau'$ \cite{kim2015guide}.

Such methods, however, rely on the feedback received from each agent being trustworthy gradient information and might fail to converge when the feedback becomes corrupt, as one single corrupt feedback can have an arbitrarily large effect. Denote the set of agents communicating corrupt gradient information, i.e., Byzantine agents, at iteration $t$ by ${\cal B}^t$, and the set of agents communicating trustworthy gradient information, i.e., trustworthy agents, at iteration $t$ by ${\cal T}^t$. At each iteration $t$, the feedback is determined as:
 \begin{equation}\label{eq:minibatchgradients}
    g_{i,t}=\left\{\begin{array}{cl}
     \nabla F_{i,t}(x_t) & \text{if }i\in {\cal T}^t, \\
        \star & \text{if } i\in{\cal B}^t,
    \end{array}\right.
\end{equation}
where the corrupt feedback $\star$ is arbitrary and is possibly chosen by an adversary, who may have full knowledge of the problem. We note that this model encompasses a large class of scenarios where the feedback can become corrupt (e.g., errors in communication or computation, corrupt data, adversarial manipulation) since we set no restrictions on $\star$.

Contrary to existing literature, we study dynamically changing sets of Byzantine agents ${\cal B}^t$ and trustworthy agents ${\cal T}^t$, where the transition of each agent from Byzantine/trustworthy state to trustworthy/Byzantine state happens probabilistically at each iteration. In particular, we define
\begin{align}
    &p_b=\mathbb{P}(i\in{\cal B}^{t+1}|i\in{\cal T}^t),\quad\forall i\in[N],\forall t,\label{eq:pb}\\
    &p_t=\mathbb{P}(i\in{\cal T}^{t+1}|i\in{\cal B}^t),\quad\forall i\in[N],\forall t,\label{eq:pt}
\end{align}
where \btt{$0<p_b<p_t<1/2$}. Accordingly, each agent's state transition over time is governed by a two-state Markov chain with transition matrix
\begin{equation}\label{eq:transitionmatrix}
   M=\begin{bmatrix}
1-p_b & p_b\\
p_t & 1-p_t
\end{bmatrix},
\end{equation}
and stationary distribution
\begin{equation}
    \pi^\star=\left[\frac{p_t}{p_t+p_b}~\frac{p_b}{p_t+p_b}\right],
\end{equation}
where state $0$ corresponds to the trustworthy state and state $1$ corresponds to the Byzantine state. We note that the exact knowledge of the transition probabilities is not necessary. We can take $p_b$ as an upper bound on the trustworthy to Byzantine transition probability, and $p_t$ as a lower bound on the Byzantine to trustworthy transition probability.

In the next section, we explain the first-order method we propose to obtain a near-optimal solution to~\eqref{eq:mainproblem} in setting defined by \eqref{eq:minibatchgradients}-\eqref{eq:pt}. 
For completeness, we end this section with a couple of standard definitions from convex analysis regarding a differentiable function $f:{\mathbb R^d\rightarrow{\mathbb R}}$.

\begin{definition}\label{def:smooth}
A differentiable function $f$ is said to be \mbox{\textbf{$\boldsymbol{L}$-smooth}} if there exists $L>0$ such that
\begin{equation}
    \|\nabla f(x_1)-\nabla f(x_2)\|\leq L \|x_1-x_2\|,
\end{equation}
for all $x_1,x_2\in \cal X$.
\end{definition}
\begin{definition}\label{def:strong}
A differentiable function $f$ is said to be \mbox{\textbf{$\boldsymbol{\mu}$-strongly convex}} if there exists $\mu>0$ such that
\begin{equation}
    \langle \nabla f(x_1)-\nabla f(x_2),x_1-x_2 \rangle\geq \mu\|x_1-x_2\|^2,
\end{equation}
for all $x_1,x_2\in \cal X$.
\end{definition}

\section{Robust \bt{Aggregating} Normalized Gradient Method (RANGE)}\label{sec:range}

\begin{algorithm}[tb]
    \caption{Robust \bt{Aggregating} Normalized Gradient Method (RANGE)}
    \begin{algorithmic}[1]
    \STATE {\bfseries Input:} Initialize $x_1\in {\cal X}$, step size $\gamma$, window size $m$, $m_0\in{\mathbb N}_0$, $T$, and $\alpha_1,\alpha_2<0.5$ s.t. $\alpha_1 m,\alpha_2 N \in \mathbb{N}_0$.
    \FOR{$t=1$ {\bfseries to} $T+m-1+m_0$}
    \STATE Broadcast $x_t$ to the agents.
    \STATE Receive $g_{i,t}$, defined in~\eqref{eq:minibatchgradients}, for $i\in[N]$.
 \IF{$t\leq m-1$}
    \STATE Set $\hat{g}_{i,t}=g_{i,t}$.
  \ELSE 
    \STATE Compute the robust mean $\hat{g}_{i,t}$ of $\{g_{i,t-\tau}\}_{\tau=0}^{m-1}$ using~\eqref{eq:median} with parameters $\alpha_1$ and $m$, for $i\in[N]$.\label{avgstep}
  \ENDIF
  \STATE Compute the robust mean $\hat{\hat{g}}_t$ of $\{\hat{g}_{i,t}\}_{i\in [N]}$ using~\eqref{eq:median} with parameters $\alpha_2$ and $N$.\label{aggregationstep}
    \STATE Compute $x_{t+1}=\Pi_{\cal X}\left(x_t-\gamma \hat{\hat{g}}_t/\|\hat{\hat{g}}_t\|\right)$.\label{updatestep}
    \ENDFOR
    \end{algorithmic}
    \label{alg:distriutedavgnormalizedgd}
\end{algorithm}

To solve Problem~\eqref{eq:mainproblem} in the Byzantine setting defined by \eqref{eq:minibatchgradients}-\eqref{eq:pt}, we propose an algorithm called Robust \bt{Aggregating} Normalized Gradient MEthod (RANGE), which is summarized in Algorithm~\ref{alg:distriutedavgnormalizedgd}. There are three main interacting ideas behind Algorithm~\ref{alg:distriutedavgnormalizedgd} to guarantee convergence and robustness:  1) \bt{temporal robust aggregation}, 2) \bt{spatial robust aggregation}, and 3) gradient normalization. \bt{Temporal robust aggregation} in Step~\ref{avgstep} of Algorithm~\ref{alg:distriutedavgnormalizedgd} aims to compute a robustified gradient for all agents by estimating 
a robust mean of a window of past gradients. The intuition behind this is that despite corruptions, the feedback received over a long period from every single agent contains a fraction of trustworthy information that the algorithm can extract to perform accurate computations. To defend against the scenarios where the robustified gradient produced by the \bt{temporal robust aggregation} step is corrupted for some of the agents (e.g., if the window only contains corrupted gradients), in Step~\ref{aggregationstep} of Algorithm~\ref{alg:distriutedavgnormalizedgd} we implement a second layer of robust mean estimation when aggregating all the agents' robustified gradients in order to eliminate those corrupted gradients. Lastly, by gradient normalization in Step~\ref{updatestep} of Algortihm~\ref{alg:distriutedavgnormalizedgd}, we restrict the aggregate gradient to only contain directional information. This prevents arbitrarily large updates in case the \bt{temporal robust aggregation} and \bt{spatial robust aggregation} steps do not sufficiently eliminate the corruptions.

Let us provide a summary of Algorithm~\ref{alg:distriutedavgnormalizedgd}.  At each iteration $t$, the central node receives the feedback $g_{i,t}$ according to~\eqref{eq:minibatchgradients} from all the agents. If $t\geq m$, it estimates a robustified gradient $\hat{g}_{i,t}$ for each agent $i\in[N]$ by performing a \bt{temporal robust aggregation} over a window of gradients $\{g_{i,t-\tau}\}_{\tau=0}^{m-1}$ using the median-based mean estimator that will be described later. If $t<m$, it simply sets $\hat{g}_{i,t}=g_{i,t}$. Then, it aggregates the robustified gradients $\{\hat{g}_{i,t}\}_{i\in[N]}$ using the same median-based mean estimator to get the robust aggregate $\hat{\hat{g}}_t$ and moves the iterate along $\hat{\hat{g}}_t/\|\hat{\hat{g}}_t\|$ with step size $\gamma$. Finally the algorithm projects the point back to the decision set $\cal X$. 

To get a good grasp of why RANGE works, let us discuss how the mechanics of each step assist the convergence of the algorithm, starting with the robust mean estimator.

\noindent\textbf{Robust mean estimator:}
  Suppose that we have a set of $k$ vectors $\{v_i\in \mathbb{R}^d\}_{i=0}^{k-1}$ that may contain corrupted values, whose identities are not known. We wish to estimate the mean of the trustworthy vectors robustly by minimizing the impact of the corrupt gradients on the mean estimate, potentially by filtering the corrupt gradients out.  We consider a simple median-based estimator applied to each coordinate $j=1,\dots,d$. First, define the coordinate-wise median as $\left[ { v}_{\sf med} \right]_j = {\sf med}\left( \{ [ v_{i}]_j \}_{i=0}^{k-1} \right)$,
where ${\sf med(\cdot)}$ computes the coordinate-wise medians. Then, our estimator is computed as the mean of the nearest \mbox{$(1-\alpha)k$} neighbors of $\left[ { v}_{\sf med} \right]_j$, where $\alpha$ is a chosen threshold parameter such that $\alpha k\in\mathbb{N}_0$. 
We propose the estimator
\begin{equation}\label{eq:median}
[\widehat{v} ]_j = \frac{1}{(1-\alpha)k} \sum_{i \in {\cal N}_j}  [{ v}_i ]_j,
\end{equation}  
where $ {\cal N}_j = \{ i \in \{0,1,\dots,k-1\}: \big| \big[ { v}_{i} -  { v}_{\sf med} \big]_j \big| \leq r_j \}$, 
such that $r_j$ is chosen to satisfy $|{\cal N}_j| = (1-\alpha)k$. 

The outcome of this estimator depends on the threshold parameter $\alpha$. If $\alpha$ is chosen such that the number of trustworthy vectors is less than $(1-\alpha)k$, then ${\cal N}_j$ will contain arbitrarily corrupted gradients and the estimate will be arbitrarily corrupted. However, if $\alpha$ is  chosen such that the number of trustworthy vectors is at least $(1-\alpha)k$, we have the following theoretical guarantees for the performance of this estimator:

\begin{proposition}\cite[Proposition 2]{turan2021resilient}\label{prop:median}
Let ${\cal H}$ be the set of trustworthy vectors and $|{\cal H}|\geq (1-\alpha)k$. Let $\overline{v}_{\cal H}$ be the mean of the trustworthy vectors. Suppose that $\max_{i \in {\cal H}} \|  v_{i} - \overline{ v}_{\cal H} \|_\infty \leq r$,
then for any $\alpha \in [0, {1}/{2})$, it holds that:
\begin{equation}
    \label{eq:median_bound}
\|\widehat{v}  - \overline{v}_{\cal H} \| \leq C_\alpha r,
\end{equation}
where
\begin{equation}
    \label{eq:calpha}C_{\alpha}=\frac{2\alpha}{1-\alpha} \left( 1 + \sqrt{\frac{(1-\alpha)^2}{1-2\alpha}}\right)\sqrt{d}.
\end{equation}
\end{proposition}

We note that the right hand side of~\eqref{eq:median_bound} can be approximated as ${\cal O}(\alpha r\sqrt{d})$ for small $\alpha$. 

\noindent \textbf{\bt{Temporal robust aggregation}:}
Following the mechanics of the robust mean estimator, two scenarios can happen every time \bt{temporal robust aggregation} is applied in Step~\ref{avgstep} of Algorithm~\ref{alg:distriutedavgnormalizedgd} to a window of $m$ latest gradients from each agent: (i) there are less than $(1-\alpha_1)m$ trustworthy gradients in the window of size $m$; (ii) there are at least $(1-\alpha_1)m$ trustworthy gradients in the window of size $m$. Under scenario (i), ${\cal N}_j$ contains arbitrarily corrupted gradients, and therefore we have to assume that the estimated mean is arbitrarily corrupted. We say that the \bt{temporal robust aggregation} \emph{fails} in this scenario. Under scenario (ii), the estimated mean is close to the true mean of the trustworthy gradients, and the error is bounded by \eqref{eq:median_bound}. Therefore if scenario (ii) happens at any iteration, instead of using a probably corrupt gradient that can be adversarial, the \bt{temporal robust aggregation} step computes a robustified gradient close to the mean of past trustworthy gradients. We say that the \bt{temporal robust aggregation} \emph{succeeds} in this scenario. Accordingly, we can view scenario (ii) as a perturbed version of the delayed gradient method, whose convergence properties have been well-studied \cite{gurbuzbalaban2017convergence}.
Note that both scenarios (i) and (ii) happen with some probability determined by $p_t$, $p_b$, $\alpha_1$ and $m$. 

\noindent \textbf{\bt{Spatial robust aggregation}:}
Similar to the \bt{temporal robust aggregation} step, two scenarios can happen every time \bt{spatial robust aggregation} is applied in Step~\ref{aggregationstep} of Algorithm~\ref{alg:distriutedavgnormalizedgd} to $N$ robustified gradients: (I) there are less than $(1-\alpha_2)N$ agents for which the \bt{temporal robust aggregation} step succeeds, (II) there are at least $(1-\alpha_2)N$ agents for which the \bt{temporal robust aggregation} step succeeds. By similar arguments as above, scenario (I) results in an arbitrarily corrupted estimate, whereas scenario (II) results in an estimate that is close to the true mean of the successfully robustified gradients, and the error is bounded by \eqref{eq:median_bound}. We note that both scenarios (I) and (II) happen with some probability determined by $\alpha_2$, $N$, and probabilities of scenarios (i) and (ii).

\noindent \textbf{Gradient normalization:}
The main idea behind the normalization step is to prevent large updates that corrupt gradients might cause. Since in the case of scenario (I), the \bt{temporal robust aggregation} and the \bt{spatial robust aggregation} steps fail to produce an aggregate gradient estimate for which the theoretical error bounds hold, we have to assume that the aggregate gradient estimate becomes arbitrarily corrupted. When this happens and corruptions get past through the two layers of defense, we limit the amount of damage caused by the corrupted aggregate gradient estimate by normalization.

In the next section, we state the convergence guarantees of RANGE for strongly convex and smooth (possibly non-convex) cost functions for the SAA setting.
\section{Convergence Properties of RANGE \\ for the SAA Setting} \label{sec:SAA}
Before presenting the convergence results, we need to state some technical assumptions.

\begin{assumption}\label{ass:subgamma}
For all $k\in [d]$ and $x\in{\cal X}$, define the random variable $f_k(x,z)\eqdef\partial_k f(x,z)-\partial_k F(x)$. We assume that for all $k\in [d]$ and $x\in{\cal X}$, $f_k(x,z)$ is a sub-gamma random variable with variance factor $\sigma$ and scale parameter $a$ for some $a\geq 0$, i.e.:
\begin{equation}
    \ \log{\underset{z\sim{\cal D}}{\mathbb{E}}[e^{\lambda f_k(x,z)}]}\leq \frac{\lambda^2\sigma^2}{2(1-a|\lambda|)},~\forall x, k, |\lambda|<\frac{1}{a}.
\end{equation}
\end{assumption}
Assumption~\ref{ass:subgamma} shows bounded moments of the loss function with respect to the data distribution. Note that sub-Gaussian/sub-exponential random variables satisfy Assumption~\ref{ass:subgamma} with $a=0$/$a=\sigma$, respectively. Therefore, Assumption~\ref{ass:subgamma} is less restrictive than sub-Gaussian/sub-exponential assumptions in the literature \cite{yin2018byzantine,yin2019defending}.
\begin{assumption}\label{ass:smooth}
The function $f(\cdot,z)$ is $L$-smooth, $\forall z\in {\cal Z}$.
\end{assumption}

In addition, when $F(\cdot)$ is strongly convex, we have the following assumption on ${\cal X}$ and the minimizer of $F(\cdot)$:
\begin{assumption}\label{ass:minimizergradient}
The parameter set ${\cal X}$ is assumed to be convex and compact with diameter $R$. Furthermore, $F(x)$ has a unique minimizer $x^\star \in {\cal X}$ satisfying $\nabla F(x^\star)=0$. 
\end{assumption}
Together with the convexity of $F$, the above assumption implies that the minimizer of $F(\cdot)$ in ${\cal X}$ is also the minimizer of $F(\cdot)$ in ${\mathbb R}^d$. We note that by selecting ${\cal X}$ as the euclidean norm ball of a large radius $R$, the assumption can be satisfied. 

Recall from Proposition~\ref{prop:median} that in order for the error bound \eqref{eq:median_bound} to hold in Algorithm~\ref{alg:distriutedavgnormalizedgd} Step~\ref{avgstep}, the robust mean estimator \eqref{eq:median} requires that at least $(1-\alpha_1)m$ vectors in $\{g_{i,t-\tau}\}_{\tau=0}^{m-1}$ are trustworthy gradients (scenario (ii) in Sec.~\ref{sec:range}). Similarly, in order for the error bound \eqref{eq:median_bound} to hold in Algorithm~\ref{alg:distriutedavgnormalizedgd} Step~\ref{aggregationstep}, the robust mean estimator \eqref{eq:median} requires that at least $(1-\alpha_2)N$ vectors in $\{\hat{g}_{i,t}\}_{i\in[N]}$ are successfully robustified (scenario (II) in Sec.~\ref{sec:range}). In order to mathematically formalize these scenarios, we define the following random variables:
\begin{itemize}
    \item $W_{i,t}=1$ if $i\in {\cal B}^t$, $0$ otherwise,
    \item $Y_{i,t}=1$ if for agent $i$, $\sum_{\tau\in[t-m+1,t]}W_{i,\tau}>\alpha_1 m$ (scenario (i)), $Y_{i,t}=0$ otherwise (scenario (ii)),
    \item $Z_t=1$ if $\sum_{i\in[N]}Y_{i,t}>\alpha_2 N$ (scenario (I)), $Z_t=0$ otherwise (scenario (II)).
\end{itemize}
Using the above definitions of random variables, when $Y_{i,t}=1$, the \bt{temporal robust aggregation} step fails to produce a robustified gradient for agent $i$. Therefore when $Z_t=1$, the algorithm fails to produce a robustified update direction as the \bt{spatial robust aggregation} step becomes contaminated. The challenge of the convergence analysis of Algorithm~\ref{alg:distriutedavgnormalizedgd} arises from studying both scenarios $Z_t=0$ and $Z_t=1$ along with their probabilities of happening. However, given the Markovian property of $\{W_{i,t}\}_{\forall t}$, $Z_t$ is not independent of the past, which presents an obstacle in the convergence analysis. To overcome this, we state the next lemma, which establishes a uniform bound on the probability that $Z_t=1$ given the network state at an earlier time instant:

\begin{lemma}\label{lem:practicalbound}
Let ${\cal S}_{t}=\{x_{t},\{\pi_{t}^i\}_{i\in[N]}\}$ denote the system state at iteration $t$, where $\pi_{t}^i$ is the distribution of the state of agent $i$ at iteration $t$. \btt{Define 
\begin{equation}\label{eq:pim0}
    \Pi_{m_0}^1=\frac{p_b+p_t(1-p_b-p_t)^{m_0}}{p_b+p_t},
\end{equation}
for all $m_0\in {\mathbb N}_0$.} Given the algorithm parameters $(m,N,\alpha_1,\alpha_2)$ \bt{and the transition matrix $M$}, for all \btt{$m_0\in {\mathbb N}_0$ such that $1/2>\alpha_1>\Pi_{m_0}^1$} and for all $t\geq m+m_0$, there exists a uniform bound on $\mathbb{P}(Z_t=1|{\cal S}_{t-m+1-m_0})=\mathbb{E}[Z_t|{\cal S}_{t-m+1-m_0}]$ independent of ${\cal S}_{t-m+1-m_0}$ such that:
\begin{equation}
    \mathbb{E}[Z_t|{\cal S}_{t-m+1-m_0}]\leq P_Z(m_0,m,N,\alpha_1,\alpha_2,M).
\end{equation}
Let $P_Z^m(m_0)\eqdef P_Z(m_0,m,N,\alpha_1,\alpha_2,M)$. Then, the above bound holds for:
\begin{align}\label{eq:PZpractical}
    P_Z^m(m_0){=}\hspace{-.4cm}\sum_{k=\alpha_2 N+1}^N\hspace{-.2cm} \binom{N}{k}(P_Y^m(m_0))^k(1{-}P_Y^m(m_0))^{(N{-}k)},
\end{align}
where
\btt{
\begin{equation}\label{eq:PYpractical}
    P_Y^m(m_0)=\exp{\left(-m(\alpha_1-\Pi_{m_0}^1)^2(p_b+p_t)\right)}.
\end{equation}}
\end{lemma}

Proof of Lemma~\ref{lem:practicalbound} can be found in Appendix~\ref{app:practicalbound}. Given a non-negative integer $m_0\in{\mathbb N}_0$, Lemma~\ref{lem:practicalbound} sets an upper bound on $\mathbb{E}[Z_t|{\cal S}_{t-m+1-m_0}]$ independent of the system state at time $t-m+1-m_0$, but only as functions of the algorithm parameters $(m,N,\alpha_1,\alpha_2)$, \bt{ the transition matrix $M$,} and $m_0$. Although Lemma~\ref{lem:practicalbound} provides a practical closed form bound, it is derived by using a Chernoff-type bound for Markov chains \cite{leon2004optimal}. It facilitates exposition of the method but it is not tight. In Appendix~\ref{app:tighterbound}, we provide a tighter bound on $P_Z^m(m_0)$. Note that by Hoeffding's inequality, we have $P_Z^m(m_0) \leq e^{-2(\alpha_2-P_Y^m(m_0))^2 N}$.
\btt{
\begin{remark}\label{rem:Pzpbpt}
It is worthwhile to discuss how the upper bound on \eqref{eq:PZpractical} depends on $p_b$ and $p_t$. By chain rule, we have that:
\begin{equation}
    \frac{d P_Z^m(m_0)}{d p_i}=\frac{d P_Z^m(m_0)}{d P_Y^m(m_0)}\times\frac{d P_Y^m(m_0)}{d p_i},~\textnormal i\in\{b,t\}.
\end{equation}
Note that $P_Z^m(m_0)$ is $1-F_B(\alpha_2N,N,P_Y^m(m_0))$, where $F_B$ is the cumulative distribution function of the binomial distribution with parameters $(N,P_Y^m(m_0))$ evaluated at $\alpha_2N$. $F_B(\alpha_2N,N,P_Y^m(m_0))$ is given by \cite{wadsworth1960introduction}:
\begin{equation}
    (N{-}\alpha_2N)\binom{N}{\alpha_2N}\int_{0}^{1{-}P_Y^m(m_0)}\hspace{-.4cm}t^{N{-}\alpha_2 N{-}1}(1{-}t)^{\alpha_2N}dt,
\end{equation}
which is decreasing with $P_Y^m(m_0)$. Accordingly, we have that ${d P_Z^m(m_0)}/{d P_Y^m(m_0)}> 0$. We also show in Appendix~\ref{app:Pypbpt} that for $m_0={\cal O}(1/(p_b+p_t))$, i.e., when $m_0$ is chosen at the order of the mixing time
\begin{equation}
    \frac{d P_Y^m(m_0)}{d p_b}>0,~~ \frac{d P_Y^m(m_0)}{d p_t}<0.
\end{equation}
Therefore, ${d P_Z^m(m_0)}/{d p_b}>0$ and ${d P_Z^m(m_0)}/{d p_t}<0$. This is in accordance with the intuition that the corruption rate increases with $p_b$ and decreases with $p_t$. 
In Figure~\ref{fig:PZ}, we plot $P_Z^m(m_0)$ for varying $p_b$ and $p_t$ while keeping the rest of the variables constant.
\end{remark}
\begin{remark}\label{rem:2}
We discuss the upper bound in Lemma~\ref{lem:practicalbound} for the edge cases of the parameters  $p_b$ and $p_t$.
Lemma~\ref{lem:practicalbound} requires $0<p_b<p_t<1/2$, i.e., the Markov chain to be ergodic. Section~\ref{sec:special3} discusses the case $p_b=p_t=0$. Moreover, for a fixed $p_t$, we can evaluate the bounds as $p_b\rightarrow 0$. Set $m_0=k_0/p_t$ and $m=k/(\alpha_1^2p_t)$ for some $k_0,k\gg 1$. Then, we have that
\begin{align}
    \underset{p_b\rightarrow 0}{\lim}P_Y^m(m_0)&=\exp{\left(-k\frac{(\alpha_1-(1-p_t)^\frac{k_0}{p_t})^2}{\alpha_1^2}\right)}\\
    &\leq \exp{(-k(1-e^{-k_0}/\alpha_1)^2)}\approx 0
\end{align}
Therefore, $P_Z^m(m_0)\approx 0$ for $k,k_0\gg 1$ as ${p_b\rightarrow 0}$.\\
In the other extreme as both $p_b,p_t\rightarrow 1/2$, we have $\Pi^1_{m_0}= 1/2$ and therefore $P_Y^m(m_0)= 1$, which results in $P_Z^m(m_0)= 1$. This is because when the corruption rate $p_b/(p_b+p_t)\rightarrow 1/2$, it becomes impossible to distinguish trustworthy information from the corrupted.
\end{remark}
}

\begin{figure}[t]
    \centering
    \includegraphics[width=.8\linewidth]{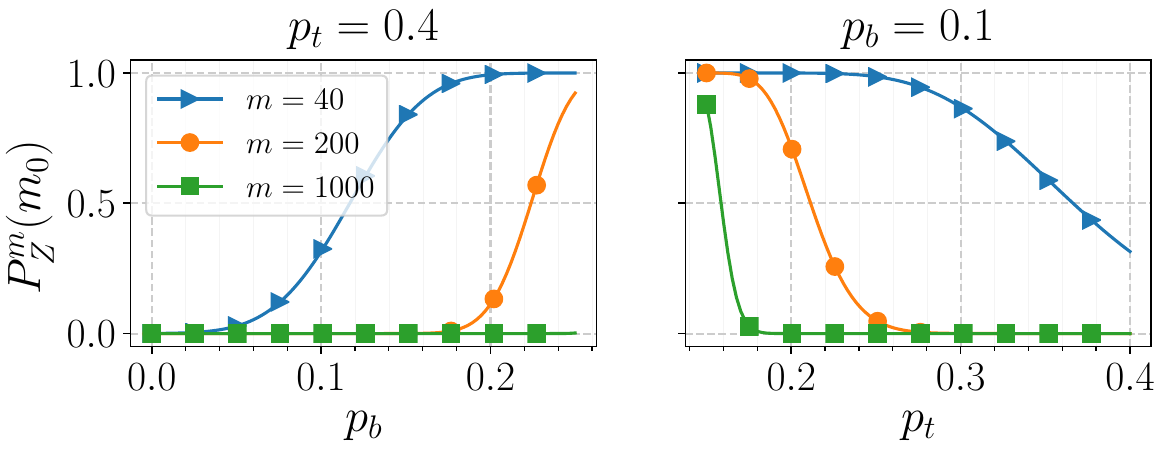}
    \caption{\btt{Plot of $P_Z^m(m_0)$ for varying $p_b$ with $p_t=0.4$ (left), and varying $p_t$ with $p_b=0.1$ (right) for window sizes $m=\{40,200,1000\}$. The rest of the parameters are kept constant at $\alpha_1=0.45$, $\alpha_2=0.3$, $N=10$, and $m_0=100$.}}
    \label{fig:PZ}
\end{figure}
\subsection{Strongly Convex and Smooth Functions}

We are now ready to present the main technical result on convergence guarantees of RANGE for a strongly convex cost function $F(\cdot)$. For the following result, we do not require strong convexity of individual cost functions $f(\cdot,z)$.
\begin{theorem}\label{thm:stronglycvx}
Let $F(\cdot)$ be $\mu$-strongly convex and Assumptions~\ref{ass:subgamma},~\ref{ass:smooth}, and~\ref{ass:minimizergradient} hold. Define the condition number as $\kappa\eqdef L/\mu$. Let $m_0\in {{\mathbb N}_0}$ be a non-negative integer \bt{such that $\Pi_{m_0}^1<1/2$}. If the algorithm parameters $(m,N,\alpha_1,\alpha_2)$ \bt{and the transition matrix $M$} satisfy
\begin{equation}\label{eq:stronglycvxcondition}
    P_Z^m(m_0)< \frac{1}{1+\kappa},
\end{equation}
then for any $T\geq 1$, the iterates produced by Algorithm~\ref{alg:distriutedavgnormalizedgd} in the SAA setting, with
\bt{\begin{equation}
    \gamma\leq\min\left\{\frac{4\sigma}{\overline{C}(m_0)\mu\sqrt{(1-\alpha_2)Nb}},\frac{\kappa R}{2}\right\},
\end{equation}}
where
\begin{equation}
    \begin{split}\label{eq:cbar}
    \overline{C}(m_0)=&1+4P_Z^m(m_0)(1+1/\kappa)(m-1+m_0)\\
    &+4\kappa(m-1)(1+C_{\alpha_1}+2C_{\alpha_2}(C_{\alpha_1}+1)),
\end{split}
\end{equation}
have the following property:
\begin{equation}\label{eq:stronglycvxresult}
\begin{split}
    \mathbb{E}&[\|x_{T+m+m_0}-x^\star\|^2]\\
        &\leq \left(\|x_{1}-x^\star\|+\gamma(m+m_0-1)\right)^2 \, \big(1-c_0(m_0)\gamma \big)^T\\
        &\quad + {\cal O}\left(\frac{1}{\sqrt{N b}}+\frac{C_{\alpha_2}}{\sqrt{b}}\right),
\end{split}
\end{equation}
where
\begin{align}
     \label{eq:c0}c_0(m_0)=&\frac{2}{\kappa R}(1-P_Z^m(m_0)(1+\kappa)),
\end{align}
and $C_{\alpha_i}$ for $i=1,2$, are given by \eqref{eq:calpha}.
\end{theorem}
\bt{\textit{Proof outline: }The proof follows by bounding the distance of $x_{t+1}$ to the optimal solution $x^\star$ in terms of $x_{t}$ via perturbed gradient analysis. We define the perturbation as $\nabla F(x_t)/\|\nabla F(x_t)\|-\hat{\hat{g}}_t/\|\hat{\hat{g}}_t\|$ and bound the norm of the perturbation in the events $Z_{t}=1$ and $Z_{t}=0$ separately.}

\bt{When $Z_t=1$, we assume the worst-case scenario such that $\hat{\hat{g}}_{t}$ is moving $x_t$ in the opposite direction of $x^\star$. When $Z_t=0$, we split the perturbation into 3 terms (Lemma~\ref{lem:errorbound} in Appendix~\ref{app:errorbound}): 1) the error due to stochastic and delayed gradients, which depends on the variance $\sigma^2$, step-size $\gamma$, window size $m$, 2) the error of the temporal robust aggregator given by the median-based mean estimator's bound \eqref{eq:median_bound}, which is proportional to the maximum distance between the stochastic gradients in a window of $m$, and 3) the error of the spatial robust aggregator given by \eqref{eq:median_bound} again, which depends on the maximum distance between the robustified gradients of the $N$ agents. To upper bound the expected value of the aforementioned maximum distances, we use Theorem 2.5 in \cite{boucheron2013concentration}, which offers a convenient bound for the expected value of the maximum of finitely many exponentially integrable random variables. \hfill$\square$}

The complete proof of Theorem~\ref{thm:stronglycvx} and the explicit constants of~\eqref{eq:stronglycvxresult} can be found in Appendix~\ref{app:stronglycvx}. According to Theorem~\ref{thm:stronglycvx}, RANGE provides convergence to a neighborhood of the optimal solution at a linear rate as long as~\eqref{eq:stronglycvxcondition} is satisfied. The neighborhood of convergence is 
\begin{equation}\label{eq:strongcvxneighborhood}
     {\cal O}\left(\frac{1}{\sqrt{N b}}+\frac{C_{\alpha_2}}{\sqrt{b}}\right),
\end{equation}
\bt{where $N$ is the number of agents and $b$ is the number of data samples used for gradient computation (i.e., mini-batch size)}.
\btt{
\begin{remark}
The convergence rate in \eqref{eq:stronglycvxresult} is governed by the term $(1-c_0(m_0)\gamma)$. Accordingly, the smaller $c_0(m_0)$, the slower the convergence.  Note that $c_0(m_0)$ given by \eqref{eq:c0} is decreasing with $\kappa$. For ill-conditioned problems with big $\kappa$, $c_0(m_0)$ is small and therefore convergence is slower. Additionally, observe that $c_0(m_0)$ is decreasing with $P_Z^m(m_0)$. In accordance with Remark~\ref{rem:Pzpbpt}, the bigger $p_b$ or the smaller $p_t$, the slower the convergence since $P_Z^m(m_0)$ gets bigger.
\end{remark}
}

\begin{remark}
In Appendix~\ref{app:stronglycvx}, we show that the only dependence of the neighborhood of convergence in~\eqref{eq:strongcvxneighborhood} on $m_0$ is through $P_Z^m(m_0)$ given by~\eqref{eq:PZpractical}. By taking $m_0\gg 1$, we minimize~\eqref{eq:pim0} to get $\Pi_{m_0}^1\approx p_b/(p_b+p_t)$, which then minimizes $P_Y^m(m_0)$ and $P_Z^m(m_0)$. Hence we get a tight asymptotic bound with respect to $m_0$ for $m_0\rightarrow\infty$.
\end{remark}


\noindent\textbf{Impact of \bt{Temporal Robust Aggregation}:} 
The \bt{temporal robust aggregation} step helps reducing the neighborhood of convergence by reducing the effective fraction of Byzantine agents at each iteration. The convergence neighborhood given by \eqref{eq:strongcvxneighborhood} consists of two terms: first term due to variance of the stochastic gradients and the second term due to Byzantine agents. In \cite{yin2018byzantine}, it is shown that in the setting with a bounded $\alpha$ fraction of Byzantine agents, no algorithm can achieve an error lower than
\begin{equation}\label{eq:lowerbound}
    \tilde{{\Omega}}\left(\frac{1}{\sqrt{Nb}}+\frac{\alpha}{\sqrt{b}}\right).
\end{equation}
In the stationary distribution of the Markov chain, the probability that an agent is Byzantine is equal to $p_b/(p_b+p_t)$, and the expected fraction of Byzantine agents in an iteration is $p_b/(p_b+p_t)$. Therefore, it is reasonable to argue that $\alpha$ in~\eqref{eq:lowerbound} is similar to $p_b/(p_b+p_t)$. For Lemma~\ref{lem:practicalbound} to hold, we need \bt{ $\alpha_1>\Pi_{m_0}^1\approx{p_b}/({p_b+p_t})$ for $m_0\gg 1$} and thus $\alpha_1$ is of the order of $\alpha$ in~\eqref{eq:lowerbound}. On the other hand, our error bound in \eqref{eq:strongcvxneighborhood} is a function of $C_{\alpha_2}$, where $C_{\alpha_2}={\cal O}(\alpha_2)$ for small $\alpha_2$. Because RANGE aims to eliminate $\alpha_2$ fraction of agents' robustified gradients via \bt{spatial robust aggregation}, it can be viewed as the effective fraction of Byzantine agents, and hence our bound is consistent with \eqref{eq:lowerbound}. Interestingly, we can set $\alpha_2$ as arbitrarily small as possible. Note that we need to satisfy \eqref{eq:stronglycvxcondition} for convergence. However, in~\eqref{eq:PZpractical}, we can select $\alpha_2$ sufficiently small for small $P_Y^m(m_0)$. But $P_Y^m(m_0)$ is given by \eqref{eq:PYpractical}, which is a function of $\alpha_1$, $m$, and $m_0$. As such, we can always set $m$ arbitrarily large such that $P_Y^m(m_0)$ is arbitrarily small for all $m_0\in{\mathbb N}_0$, and hence we can select $\alpha_2$ arbitrarily small to satisfy \eqref{eq:stronglycvxcondition}. As a result, by employing the \bt{temporal robust aggregation} step before \bt{spatial robust aggregation}, which is beneficial only when agents' states change over time, we reduce the effective fraction of Byzantine agents from $p_b/(p_b+p_t)$ to $\alpha_2$. This setup shows that the lower bound~\eqref{eq:lowerbound}  does not hold for the proposed Markovian setting as a result of \bt{temporal robust aggregation}. 
\bt{
\begin{remark}
Without any corruption, RANGE suffers an error of ${\cal O}(1/\sqrt{Nb})$. Even if we use unbiased stochastic gradients with diminishing step-sizes, this error is unavoidable due to the normalization step. This is because as observed by \cite{cutkosky2020momentum}, if $\nabla F_i(x_t)=\nabla F(x_t)+\eta$ for very small $\eta$, it might be that $\nabla F_i(x_t)/\|\nabla F_i(x_t)\|$ is very far from $\nabla F(x_t)/\|\nabla F(x_t)\|$, especially when $\|\nabla F(x_t)\|$ is close to 0. On the other hand, the normalization step is the key feature of our algorithm to achieve robustness. Due to the stochasticity of the Markovian model, there is a probability ( $\leq P_Z^m(m_0)$) that the first two layers of robust aggregation fail. In this case, the overall aggregate can be arbitrarily corrupted and we require normalization to defend against such cases. Figure~\ref{fig:normalization} numerically demonstrates the importance of normalization by simulating RANGE with and without normalization for the linear regression problem in Section~\ref{sec:linearreg}.
\end{remark}}
\begin{figure}
    \centering
    \includegraphics[width=.75\linewidth]{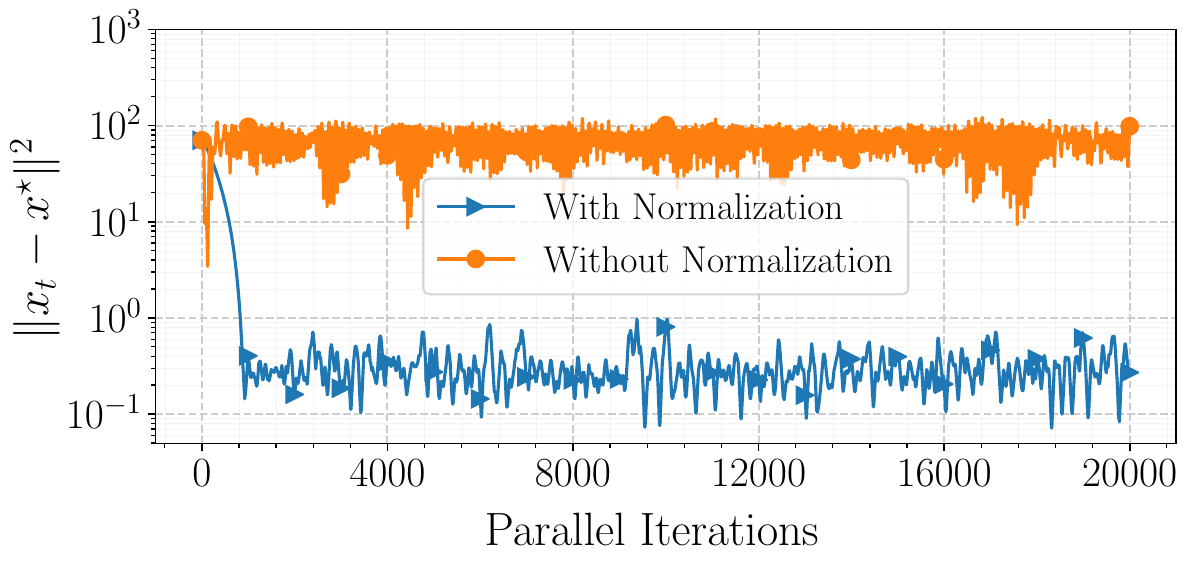}
    \vspace{-.2cm}
    \caption{\bt{RANGE with and without the normalization step for the linear regression problem in Section~\ref{sec:linearreg}.}}
    \label{fig:normalization}
\end{figure}

Observe that choice of $m$, $\alpha_1$, and $\alpha_2$ plays an important role, as they determine whether~\eqref{eq:stronglycvxcondition} holds or not. Next, we discuss how to pick $m$, $\alpha_1$, and $\alpha_2$ in practice.

\noindent \textbf{Choices of $m$, $\alpha_1$, and $\alpha_2$:}
For a given transition matrix $M$, there always exists a set of parameters $m$, $\alpha_1$, and $\alpha_2$ such that \eqref{eq:stronglycvxcondition} is satisfied. Finding a principled way to select the set of optimal parameters is currently an open question. Instead, next we describe an implementable closed form expression for the minimum window size as a function of $\alpha_1$, $\alpha_2$, $p_t$, and $p_b$. Using Hoeffding's inequality on \eqref{eq:PZpractical}, we rewrite \eqref{eq:stronglycvxcondition} as:
\begin{equation}\label{eq:hoeffdingpz}
    P_Z^m(m_0)\leq \exp (-2(\alpha_2-P_Y^m(m_0))^2N)<\frac{1}{1+\kappa}.
\end{equation}
This gives us the condition on $\alpha_2$\footnote{\label{ft:alpha2}This is the condition on $\alpha_2$ for the Hoeffding bound \eqref{eq:hoeffdingpz} to hold, rather than the exact inequality in \eqref{eq:PZpractical}. Consequently, it results in an additive $\sqrt{\log(1+\kappa)/(2N)}$ term that is independent of $P_Y^m(m_0)$. However, this does not contradict our statement that we can set $\alpha_2$ arbitrarily small by reducing $P_Y^m(m_0)$ with a large window $m$, since that statement is based on the exact form in \eqref{eq:PZpractical} rather than the Hoeffding bound \eqref{eq:hoeffdingpz}.}:
\begin{equation}\label{eq:a2lb}
    \alpha_2>P_Y^m(m_0)+\sqrt{\frac{\log(1+\kappa)}{2N}}.
\end{equation}
\bt{
Considering \eqref{eq:PYpractical}, given $m_0\in {\mathbb N_0}$, if 
\begin{equation}
    \exp{\left(-m(\alpha_1-\Pi_{m_0}^1)^2(p_b+p_t)\right)}<\alpha_2-\sqrt{\frac{\log(1+\kappa)}{2N}}
\end{equation}
holds, then \eqref{eq:stronglycvxcondition} is satisfied. Additionally, we require that $\alpha_1<0.5$ for the algorithm's input so that the robust mean estimator succeeds, which requires $\Pi_{m_0}^1<0.5$. This also gives us a lower bound on $m_0$:
\begin{equation}\label{eq:m0lb}
    m_0>\frac{\log(p_t-p_b)-\log(2p_t)}{\log(1-p_b-p_t)} 
\end{equation}
We can rearrange \eqref{eq:a2lb} to get the minimum window size that is sufficient for convergence as a function of $\alpha_1$, $\alpha_2$, $p_t$, $p_b$, and any choice of $m_0$ that satisfies \eqref{eq:m0lb}:
\begin{corollary}\label{cor:minwindow}
For all $m_0\in \mathbb{N}_0$ satisfying \eqref{eq:m0lb}, if $\alpha_2>\sqrt{\log(1+\kappa)/(2N)}$, $\alpha_1>\Pi_{m_0}^1$, and
\begin{equation}\label{eq:minwindow}
    m>-\frac{\log(\alpha_2-\sqrt{(1+\kappa)/(2N)})}{(\alpha_1-\Pi_{m_0}^1)(p_b+p_t)}
\end{equation}
then \eqref{eq:stronglycvxcondition} holds.
\end{corollary}}

Corollary~\ref{cor:minwindow} is convenient in practice for selecting $\alpha_1$, $\alpha_2$, and $m$. Given $p_b$, $p_t$, and $\kappa$, one picks \bt{$m_0$ and $\alpha_1$ such that $\Pi_{m_0}^1<\alpha_1<0.5$} and $\alpha_2$ such that $\sqrt{\log(1+\kappa)/(2N)}<\alpha_2<0.5^{\ref{ft:alpha2}}$ and $\alpha_2 N\in{\mathbb N}_0$. Then, the window size is $m$ is picked such that $\alpha_1 m \in\mathbb{N}_0$ and~\eqref{eq:minwindow} holds. With these parameter choices, \eqref{eq:stronglycvxcondition} is satisfied. Note that the $(p_t+p_b)$ term in the denominator of~\eqref{eq:minwindow} corresponds to the spectral gap of the Markov chain. Therefore, smaller spectral gap, which also implies larger relaxation and mixing time~\cite{levin2017markov}, results in a larger minimum window size. Intuitively, we select the window size at the order of the mixing time so that the Markov chain gets close to its stationary distribution and the Byzantine agents transition into trustworthy state. This allows the \bt{temporal robust aggregation} step to successfully produce a robustified gradient by extracting the trustworthy information.

\subsection{Smooth (Possibly Non-Convex) Functions}
Next, we study the convergence of RANGE for smooth (possibly non-convex) cost functions. For this problem class, we need the following assumption on ${\cal X}$:
\begin{assumption}\label{ass:noncvxball}
Problem~\eqref{eq:mainproblem} is unconstrained, i.e., ${\cal X}={\mathbb R}^d$.
\end{assumption}
The next theorem states the convergence guarantees of RANGE for smooth $F(\cdot)$:

\begin{theorem}\label{thm:noncvx}
Let Assumptions \ref{ass:subgamma}, \ref{ass:smooth}, and \ref{ass:noncvxball} hold. Choose step-size $\gamma={\gamma_0}/{\sqrt{T}}$ with $\gamma_0>0$. Let $m_0\in {{\mathbb N}_0}$ be a non-negative integer \bt{such that $\Pi_{m_0}^1<1/2$}. If the algorithm parameters $(m,N,\alpha_1,\alpha_2)$ \bt{and the transition matrix $M$} satisfy 
\begin{equation}\label{eq:noncvxcondition}
   P_Z^m(m_0) < 1/2,
\end{equation}
then for any $T \geq 1$, the iterates $\{ x_t \}_{t=m+m_0}^{T+m-1+m_0}$ produced by Algorithm~\ref{alg:distriutedavgnormalizedgd} in the SAA setting satisfy
\begin{equation}\label{eq:thmnoncvxfinalresult}
\begin{split}
               &\frac{1}{T}\sum_{t=m+m_0}^{T+m-1+m_0}\mathbb{E}[\|\nabla F(x_t)\|]\\
               &\leq\frac{F(x_1)-F(x^\star)}{\sqrt{T}\gamma_0(1-2P_Z^m(m_0))}+\frac{\overline{C}(m_0)\gamma_0}{\sqrt{T}}+{\cal O}\left(\frac{1}{T}\right)\\
        &+{\cal O}\left(\frac{1}{\sqrt{N b}}+\frac{C_{\alpha_2}}{\sqrt{b}}\right),
\end{split}
\end{equation}
where
\begin{equation}
\begin{split}\label{eq:cbarnoncvx}
       &\overline{C}(m_0)=L\Big(1/2+4(m-1+m_0)P_Z^m(m_0)\\
       &+2(m-1)(1+C_{\alpha_1}+2C_{\alpha_2}(C_{\alpha_1}+1))\Big),
\end{split}
\end{equation}
and $C_{\alpha_i}$ for $i=1,2$, are given by \eqref{eq:calpha}.
\end{theorem}
The proof of Theorem~\ref{thm:noncvx} is similar to that of Theorem~\ref{thm:stronglycvx}, and the complete proof and the explicit constants of~\eqref{eq:thmnoncvxfinalresult} can be found in Appendix~\ref{app:noncvx}. According to Theorem~\ref{thm:noncvx}, RANGE produces a point $\tilde{x}\in\{x_{m+m_0},\dots,x_{T+m-1+m_0}\}$ such that
\begin{equation}\label{eq:noncvxneighborhood}
\begin{split}
        \mathbb{E}\left[\|\nabla F(\tilde{x})\|\right]\leq&{\cal O}\left(\frac{1}{\sqrt{N b}}+\frac{C_{\alpha_2}}{\sqrt{b}}+\frac{1}{\sqrt{T}}\right).
\end{split}
\end{equation}
Note that when $T\rightarrow\infty$, the right-hand side of the above inequality is the same as \eqref{eq:strongcvxneighborhood}.
\bt{
\begin{remark}
We mention that Theorem~\ref{thm:noncvx} is valid for smooth convex cost functions and strongly convex cost functions with an unbounded parameter set as well. For convex functions with a bounded set, we can add regularization terms to make them strongly convex, in which case the guarantees of Theorem~\ref{thm:stronglycvx} hold. Adding regularization terms is a typical technique used in optimization, called dual smoothing \cite{nesterov2005smooth}.
\end{remark}}
In the next section, we state the convergence guarantees of RANGE for strongly convex and smooth (possibly non-convex) cost functions for the SA setting.

\section{Convergence Properties of RANGE \\
for the SA Setting}\label{sec:SA}
Theorems~\ref{thm:stronglycvx} and \ref{thm:noncvx} state convergence guarantees of RANGE for the SAA setting. The next theorem states the convergence result for the SA setting for strongly convex cost functions.

\begin{theorem}\label{thm:stronglycvxiid}
Let $F(\cdot)$ be $\mu$-strongly convex and Assumptions~\ref{ass:subgamma}, \ref{ass:smooth}, and \ref{ass:minimizergradient} hold. Define the condition number as $\kappa\eqdef L/\mu$. Let $m_0\in {{\mathbb N}_0}$ be a non-negative integer \bt{such that $\Pi_{m_0}^1<1/2$}. If the algorithm parameters $(m,N,\alpha_1,\alpha_2)$ \bt{and the transition matrix $M$} satisfy
\begin{equation}\label{eq:stronglycvxiidcondition}
    P_Z^m(m_0)< \frac{1}{1+\kappa},
\end{equation}
then for any $T\geq 1$, the iterates produced by Algorithm~\ref{alg:distriutedavgnormalizedgd} in the SA setting, with
\begin{equation}
    \gamma\leq\min\left\{\frac{4\sigma}{\overline{C}(m_0)\mu\sqrt{(1-\alpha_2)N(1-\alpha_1)mb}},\frac{\kappa R}{2}\right\},
\end{equation}
have the following property:
\begin{equation}\label{eq:stronglycvxresultiiddata}
    \begin{split}
        \mathbb{E}&[\|x_{T+m+m_0}-x^\star\|^2]\leq\\
        &\left(\|x_{1}-x^\star\|+\gamma(m+m_0-1)\right)^2(1-c_0(m_0)\gamma)^T\\
        &+{\cal O}\left(\frac{1}{\sqrt{(1{-}\alpha_1)mNb}}{+}\frac{C_{\alpha_2}{+}C_{\alpha_1}(1{+}C_{\alpha_2})}{\sqrt{b}}\right),
    \end{split}
\end{equation}
where $c_0(m_0)$, $\overline{C}(m_0)$ and $C_{\alpha_i}$ for $i=1,2$, are given by \eqref{eq:c0}, \eqref{eq:cbar}, and \eqref{eq:calpha}.
\end{theorem}
Proof of Theorem~\ref{thm:stronglycvxiid} and the explicit constants of \eqref{eq:stronglycvxresultiiddata} can be found in Appendix~\ref{app:stronglycvxiid}. According to Theorem~\ref{thm:stronglycvxiid}, RANGE provides convergence to a neighborhood of the optimal solution at a linear rate, where the neighborhood of convergence is
\begin{equation}\label{eq:stronglycvxerrororderiid}
    {\cal O}\left(\frac{1}{\sqrt{(1{-}\alpha_1)mNb}}{+}\frac{C_{\alpha_2}{+}C_{\alpha_1}(1{+}C_{\alpha_2})}{\sqrt{b}}\right).
\end{equation}
Comparing the above result to \eqref{eq:strongcvxneighborhood}, there are two impacts of using the SA setting instead of the SAA setting:
\begin{enumerate}[wide, labelindent=0pt,topsep=.5mm]
\item The error due to variance of the stochastic gradients reduces by a factor of ${\cal O}(\sqrt{(1-\alpha_1)m})$,
\item The error due to Byzantine agents increases by a factor of ${\cal O}(1+C_{\alpha_1}+C_{\alpha_1}/C_{\alpha_2})$.
\end{enumerate}
When agents use new samples at each iteration in the SA setting, \bt{temporal robust aggregation} results in a variance reduction, since it estimates the mean of $(1-\alpha_1)m$ independent minibatch gradients. However, this comes at the cost of the higher error caused by the Byzantine agents. 
Given these two counteracting impacts of the SA setting on the error, the order of error in \eqref{eq:stronglycvxerrororderiid} is less than \eqref{eq:strongcvxneighborhood} if
\begin{equation}
    C_{\alpha_1}<\frac{1}{\sqrt{N}(1+C_{\alpha_2})}.
\end{equation}
Therefore, RANGE performs  better in the SA setting compared to the SAA setting if $\alpha_1\ll 1$, which is possible if $p_b\ll p_t$. Otherwise, the benefit of variance reduction provided by \bt{temporal robust aggregation} is dominated by the damage caused by the Byzantine agents.

The next theorem states the convergence result of RANGE for the SA setting for non-convex cost functions.
\begin{theorem}\label{thm:noncvxiid}
 Let Assumptions \ref{ass:subgamma}, \ref{ass:smooth}, and \ref{ass:noncvxball} hold. Choose step-size $\gamma={\gamma_0}/{\sqrt{T}}$ with $\gamma_0>0$. Let $m_0\in {{\mathbb N}_0}$ be a non-negative integer \bt{such that $\Pi_{m_0}^1<1/2$}. If the algorithm parameters $(m,N,\alpha_1,\alpha_2)$ \bt{and the transition matrix $M$} satisfy 
\begin{equation}
   P_Z^m(m_0) < 1/2,
\end{equation}
then for any $T \geq 1$, the iterates $\{ x_t \}_{t=m+m_0}^{T+m-1+m_0}$ produced by Algorithm~\ref{alg:distriutedavgnormalizedgd} in the SA setting satisfy
\begin{equation}\label{eq:thmnoncvxfinalresultiiddata}
\begin{split}
               &\frac{1}{T}\sum_{t=m+m_0}^{T+m-1+m_0}\mathbb{E}[\|\nabla F(x_t)\|]\\
               &\leq\frac{F(x_1)-F(x^\star)}{\sqrt{T}\gamma_0(1-2P_Z^m(m_0))}+\frac{\overline{C}(m_0)\gamma_0}{\sqrt{T}}+{\cal O}\left(\frac{1}{T}\right)\\
       &+{\cal O}\left(\frac{1}{\sqrt{(1{-}\alpha_1)mNb}}{+}\frac{C_{\alpha_2}{+}C_{\alpha_1}(1{+}C_{\alpha_2})}{\sqrt{b}}\right),
\end{split}
\end{equation}
where $\overline{C}(m_o)$ and $C_{\alpha_i}$ for $i=1,2$, are given by \eqref{eq:cbarnoncvx} and~\eqref{eq:calpha}.
\end{theorem}
Proof of Theorem~\ref{thm:noncvxiid} and the explicit constants of \eqref{eq:thmnoncvxfinalresultiiddata} can be found in Appendix~\ref{app:noncvxiid}. Note that when $T\rightarrow\infty$, the right hand side of \eqref{eq:thmnoncvxfinalresultiiddata} is the same as \eqref{eq:stronglycvxerrororderiid}.


\section{Special Cases}\label{sec:special}
\subsection{Window Size $m=1$}
When we select the window size $m$ to be 1, we have to set $\alpha_1=0$ since $\alpha_1 m \in{\mathbb N}_0$ and $\alpha_1<0.5$. This means that we skip the \bt{temporal robust aggregation} step
and set the robustified gradient $\hat{g}_{i,t}=g_{i,t}$. In this case, the counterpart of Lemma~\ref{lem:practicalbound} with $m=1$ gives:
\begin{align}\label{eq:pzm1}
    P_Z^1(m_0){\leq}\hspace{-.4cm}\sum_{k=\alpha_2 N+1}^N\hspace{-.2cm} \binom{N}{k}(P_Y^1(m_0))^k(1{-}P_Y^1(m_0))^{(N{-}k)},
\end{align}
where
\begin{align}
   \nonumber P_Y^1(m_0)&{=}\underset{i\in[N],t}{\max}\mathbb{P}(Y_{i,t}{=}1|{\cal S}_{t-m_0}){=}\underset{i\in[N],t}{\max}\mathbb{P}(W_{i,t}{=}1|{\cal S}_{t-m_0})\\
    &=\frac{p_b+p_t(1-p_b-p_t)^{m_0}}{p_b+p_t}.\label{eq:pym1}
\end{align}
Accordingly, Theorems~\ref{thm:stronglycvx}, \ref{thm:noncvx}, \ref{thm:stronglycvxiid}, and \ref{thm:noncvxiid} hold with $m=1$ and $P_Z^1(m_0)$ given by \eqref{eq:pzm1} and \eqref{eq:pym1}. Note that $P_Y^1(m_0)\rightarrow p_b/(p_b+p_t)$ as $m_0\rightarrow \infty$. 
\subsection{Independent Random Corruption}
A special case of the agents' state transition occurs when $p_b+p_t=1$. We get $M=[{\pi^\star}^T~{\pi^\star}^T]^T$, and hence the state of an agent at $t+1$ is independent of the state at $t$. In this case, an agent becomes Byzantine and sends corrupted gradient information randomly with probability $p_b$ at all iterations. Hence, we can state the counterpart of Lemma~\ref{lem:practicalbound} without the need to condition on a previous time instant, i.e.,
\begin{equation}\label{eq:pznonmarkovian}
   P_Z^m\eqdef\mathbb{E}[Z_t] =\hspace{-.3cm}\sum_{k=\alpha_2N +1}^{N}\hspace{-.1cm}\binom{N}{k}(P_Y^m)^k(1-P_Y^m)^{(N-k)},
\end{equation}
where
\begin{equation}\label{eq:pynonmarkovian}
    P_Y^m=\sum_{j=\alpha_1m+1}^{m}\binom{m}{j}p_b^jp_t^{m-j}.
\end{equation}
Accordingly, Theorems~\ref{thm:stronglycvx}, \ref{thm:noncvx}, \ref{thm:stronglycvxiid}, and \ref{thm:noncvxiid} hold with $m_0=0$ and $P_Z^m(m_0)$ replaced by $P_Z^m$ given by \eqref{eq:pznonmarkovian} and \eqref{eq:pynonmarkovian}.
\btt{
\subsection{Recovering the Static Corruption Model}\label{sec:special3}
The static setting where agents do not change states is recovered by letting $p_b=p_t=0$. Recall from Remark~\ref{rem:2} that Lemma~\ref{lem:practicalbound} can not be used in this case as it requires the Markov chain to be ergodic. However, if we assume that $p_b/(p_b+p_t)$ fraction of the agents are initially corrupted, our results recover the static setting by choosing $\alpha_2=p_b/(p_b+p_t)$ and $m_0=0$ to ensure $P_Z^m(0)=0$ for any $m$. Furthermore, we can remove the temporal robust aggregation step by letting $m=1$ and $\alpha_1=1$, since it will not eliminate any corruption and does not provide any variance reduction for the SAA setting.}
\section{Numerical Experiments}\label{sec:numerical}
\begin{figure}
    \centering
    \includegraphics[width=\linewidth]{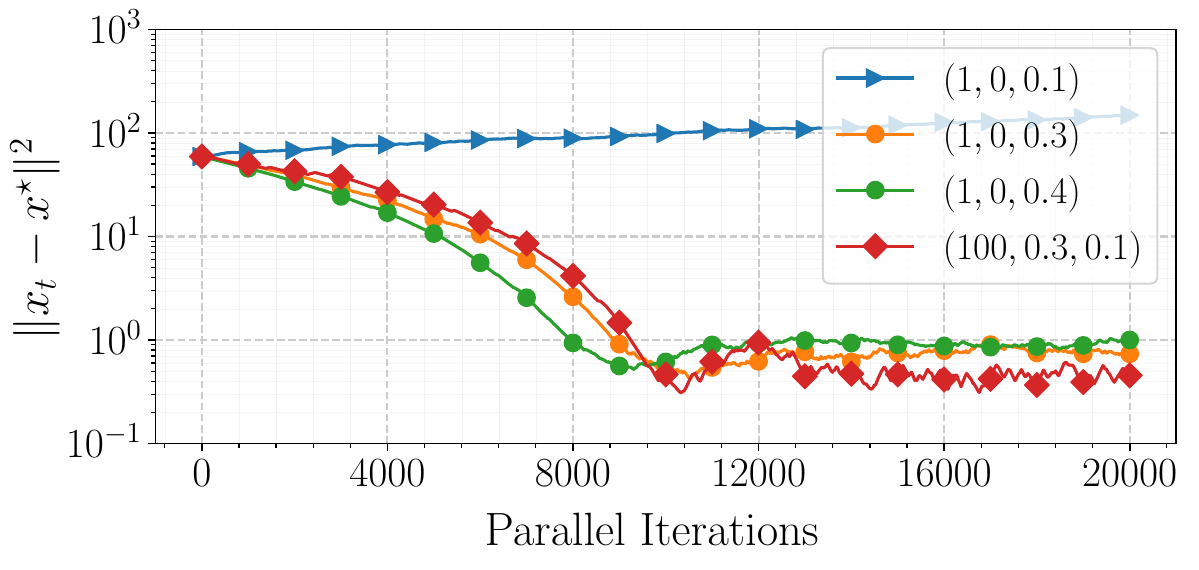}
    \caption{Convergence performance of RANGE in linear regression with four configurations of $(m,\alpha_1,\alpha_2)$ for $p_t=0.1$, $p_b=0.025$.}
    \label{fig:linear_regression}
\end{figure}
In this section, we present numerical evidence supporting our theoretical results and demonstrating the efficacy of RANGE. The first experiment is a simple linear regression with synthetic data to illustrate the benefits of the \bt{temporal robust aggregation} step of RANGE and to compare the SAA and the SA settings. The second experiment is an image classification task on the EMNIST dataset~\cite{cohen2017emnist} using a neural network to compare the performance of RANGE to existing distributed optimization algorithms in practical non-convex tasks for the SAA setting. Both experiments were performed on a laptop computer with Intel$^\textnormal{\textregistered}$ Core$^\textnormal{TM}$ i7-8750H CPU (6$\times$2.20 GHz) and 16 GB DDR4 2666MHz RAM.
\subsection{Linear Regression with Synthetic Data}

\label{sec:linearreg}

We consider the following stochastic optimization problem:
\begin{equation}
    x^\star=\underset{x\in{\cal X}}{\arg\min}~\underset{V,Y}{\mathbb{E}}[|Y-V^Tx|],
\end{equation}
where $V\in {\mathbb R}^d$ is the random vector corresponding to the data points and $Y\in {\mathbb R}$ is the random variable corresponding to the associated label values or outputs. We let $d=100$ and constructed the solution vector $x^\star$ by sampling a random point from the interior of the $d$-ball with radius $R=10$. In the SAA setting, the goal is to solve the following deterministic optimization problem
\begin{equation}
    \underset{x\in{\cal X}}{\min}~\|y-vx\|^2,
\end{equation}
where $v\in\mathbb{R}^{B\times d}$ is a matrix containing the $B$ data vectors in its rows and $y\in\mathbb{R}^B$ is the vector containing the $B$ associated label values or outputs. 
We let $B=1000$ and randomly generated the entries of $v$ from ${\cal N}(0,1)$. We distributed the data points equally among $N=10$ agents. For all $i\in[B]$, we generated the outputs $y$ according to $y_i=v_i x^\star +\xi_i$, where $\xi_i\sim{\cal N}(0,R^2)$ is the noise and $v_i$ is the $i$'th row of~$v$.

\bt{The main goal of the experiment is to demonstrate: 1) how $\alpha_2$ affects the robustness and the performance of the algorithm, and 2) how temporal robust aggregation can help us pick a better $\alpha_2$ and improve the performance. We note that there are two effects of $\alpha_2$ on the performance: 1) In the spatial robust aggregation step, the algorithm eliminates $\alpha_2$ fraction of the robustified gradients $\{\hat{g}_{i,t}\}_{i\in[N]}$ and therefore provides robustness in case there are less than $\alpha_2$ fraction of corrupted  $\{\hat{g}_{i,t}\}_{i\in[N]}$, and 2) by aggregating $(1-\alpha_2)N$ number of agents' robustified gradients, it reduces the variance of the stochastic gradients. In order to guarantee robustness, it is desirable to increase $\alpha_2$ so that we do not add the corrupted $\{\hat{g}_{i,t}\}_{i\in[N]}$ to the aggregate. On the other hand, increasing $\alpha_2$ results in aggregating less number of agents' gradients, which in turn results in a higher variance of the stochastic gradients. Given these counteracting impacts of $\alpha_2$, it is not desirable to choose it too small or too big.}

We ran RANGE for 20k iterations $p_t=0.1$ and $p_b=0.025$ using four configurations of $(m,\alpha_1,\alpha_2)$. At each iteration, we picked the corrupt gradient as $2\|\nabla F(x_t)\|({x^\star-x_t})/{\|x^\star-x_t\|}$. In Figure~\ref{fig:linear_regression}, we plot the convergence behaviour of RANGE for all four configurations.  When $m=1$ and $\alpha_2=0.1$, we observe that the iterates diverge. Since $0.1<p_b/(p_b+p_t)=0.2$, the expected value of the fraction of Byzantine agents at each iteration is larger than $\alpha_2$, and hence the aggregate gradient estimate becomes corrupted most of the time. On the contrary, setting $\alpha_2=0.3$ or $\alpha_2=0.4$ provides robustness and RANGE converges. However, we observe that the configuration with $\alpha_2=0.3$ performs slightly better than the one with $\alpha_2=0.4$. This is because smaller $\alpha_2$  aggregates more agents' gradients, which results in a larger variance reduction. All in all, while selecting a smaller $\alpha_2$ provides variance reduction, it reduces the robustness of RANGE by including more agents at the \bt{spatial robust aggregation} step.

On the other hand, the configuration with $m=100$, $\alpha_1=0.3$ and $\alpha_2=0.1$ outperforms the rest. When we utilize the \bt{temporal robust aggregation} step, it effectively reduces the expected value of the fraction of Byzantine agents at each iteration. Consequently, we can select a smaller $\alpha_2$ in order to benefit from larger variance reduction while still being robust thanks to \bt{temporal robust aggregation}.

\begin{figure}[t]
    \centering
    \includegraphics[width=1\linewidth]{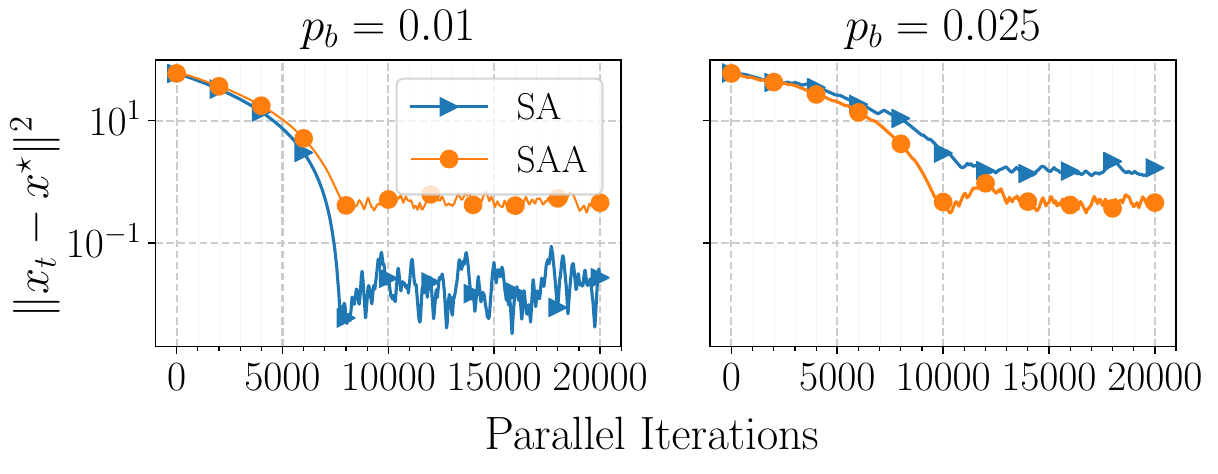}
    \caption{Comparison of the SA and the SAA settings in linear regression.}
    \label{fig:savssa}
\end{figure}

Next, we study the SA setting by re-sampling $B$ new data points at each iteration. Keeping $p_t=0.1$ constant, we simulate a high corruption rate with $p_b=0.025$ and a low corruption rate with $p_b=0.01$. We set $\alpha_1=0.3$ when $p_b=0.025$ and $\alpha_1=0.1$ when $p_b=0.01$. We fixed the window size $m=100$ and $\alpha_2=0.1$ for both cases. In Figure~\ref{fig:savssa} we plot the convergence behavior of RANGE in the SA and the SAA settings for both corruption rates. We observe that when the corruption rate is low, RANGE performs better in the SA setting due to variance reduction provided by \bt{temporal robust aggregation}. However, when the corruption rate is high, RANGE performs worse in the SA setting as the damage caused by the Byzantine agents dominates.
\begin{figure*}[t]
    \centering
    \includegraphics[width=.9\linewidth]{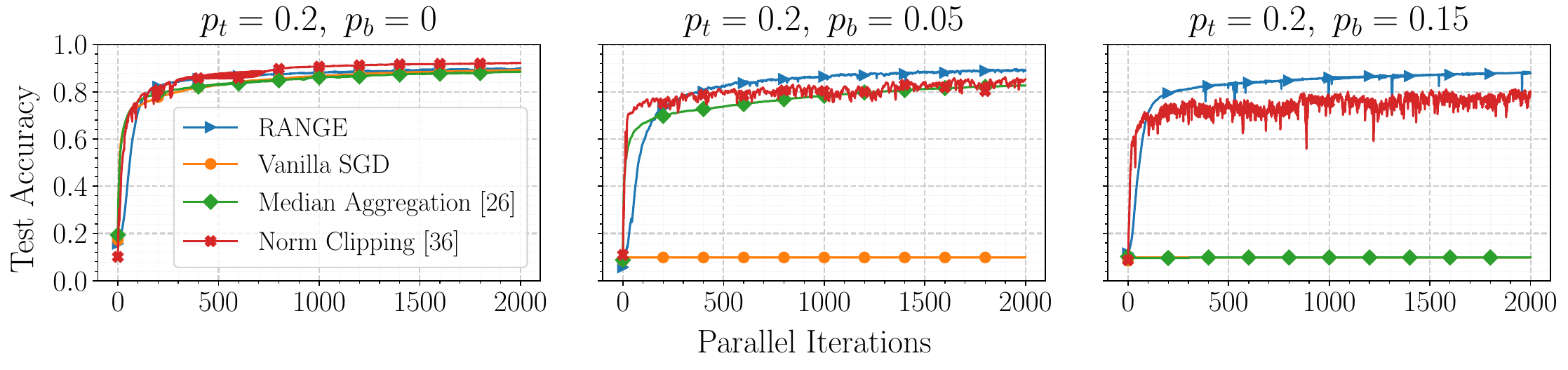}
    \caption{Training neural networks for image classification on the EMNIST dataset. Four distributed optimization algorithms are compared under $p_t=0.2$ and $p_b=0$ (left), $p_b=0.015$ (middle), and $p_b=0.15$ (right) using test accuracy of the trained network as metric. The legend is shared among all plots.}
    \label{fig:mnist}
    \vspace{-.5cm}
\end{figure*}

\subsection{Image Classification with EMNIST Dataset}

In this study, we experiment with the image classification task on the EMNIST dataset \cite{cohen2017emnist} in the SAA setting. We train a feed-forward neural network with two hidden layers and 64 neurons at each hidden layer. We partition the data into $N=200$ equal sizes, representing the data at $N$ agents. The batch size for gradient computation is set to $b=300$.

We train the neural network using (a) RANGE, (b) vanilla SGD, (c) median aggregation \cite{yin2018byzantine}, and (d) norm clipping \cite{sun2019can}. We note that although the median aggregation method in \cite{yin2018byzantine} and the norm clipping method in \cite{sun2019can} are developed for the setting with a bounded fraction of Byzantine agents, we implement them in the Markovian Byzantine agent setting because no existing work studies the same setup as ours. For transition probabilities $p_t=0.2$ and $p_b\in\{0,~0.05,~0.15\}$, we train three networks with learning rates $0.1$, $0.01$, and $0.001$, and pick the best-performing one. We let $m=50$, $\alpha_1=0.25$, $\alpha_2=0.2$ for RANGE when $p_b=0$ and $p_b=0.05$; $m=50$, $\alpha_1=0.45$, $\alpha_2=0.3$ when $p_b=0.15$. We set the threshold for norm clipping to be 10 as it is shown to perform well in \cite{sun2019can}. We simulate corruption by simply inverting and boosting the magnitude of the gradient, i.e., by setting $\star=-c\nabla F_{i,t}(x_t)$ in \eqref{eq:minibatchgradients}, where $c$ is sampled uniformly from $[5,15]$ at each iteration.

In Figure~\ref{fig:mnist} we plot the test accuracy of the models during training. For $p_b=0$, all algorithms successfully train the neural network as expected. For $p_b=0.05$, vanilla SGD has negligible performance as it is not robust to corruption. On the other hand, norm clipping and median aggregation methods have satisfactory performance. Nonetheless, RANGE outperforms norm clipping and median aggregation algorithms by margins of $3.7\%$ and $6.3\%$, respectively. For $p_b=0.15$, we observe that the median aggregation method also fails. The median is no longer robust since it is corrupted if more than half of the agents become Byzantine at any iteration, which frequently happens when $p_b/(p_b+p_t)$ is high. Norm clipping is still robust, however, RANGE beats it by a margin of $8.3\%$.
\section{Conclusions and Future Work}
We introduced a distributed optimization algorithm, named RANGE, that is provably robust to Byzantine failures. By modeling each agent's state transition trajectory over time, namely from trustworthy to Byzantine and vice versa, as a two-state Markov chain, we allow all the agents to be prone to failure. RANGE is based on three ideas: 1) \bt{temporal robust aggregation}, which computes a robustified gradient for each agent by estimating a robust mean of a window of past gradients, 2) \bt{spatial robust aggregation}, which computes a robust mean of all the agents' robustified gradients to estimate the aggregate gradient, and  3) gradient normalization, which restricts the aggregate gradient to only contain directional information and therefore prevents arbitrarily large updates that corrupt gradients might cause. We prove that for strongly convex and smooth non-convex cost functions RANGE achieves \bt{convergence to a neighborhood of a stationary point, where the neighborhood depends on the variance of the stochastic gradients and the corruption rate}. Numerical experiments on linear regression and image classification on EMNIST dataset demonstrate the robustness and efficacy of RANGE.

\bt{The temporal robust aggregation step of RANGE requires the storage of $mN$ number of $d$-dimensional vectors, which can be seen as a trade-off between memory and robustness. It would be an interesting future direction to develop a robust mean estimator that can save on memory costs.} Additionally, future work should study accelerated formulations of RANGE. As shown in~\cite{cutkosky2020momentum} momentum-based approaches accelerate normalized SGD. Whether RANGE benefits from momentum terms while maintaining its robustness properties is an open question. Lastly, future work should study the robustness properties of statistically preconditioning approaches for large-scale distributed optimization methods~\cite{hendrikx2020statistically,yuan2020convergence,dvurechensky2021hyperfast}.

\bibliographystyle{IEEEtran}
\bibliography{references}
\newpage
\begin{IEEEbiography}[{\includegraphics[width=1in,height=1.21in,clip,keepaspectratio]{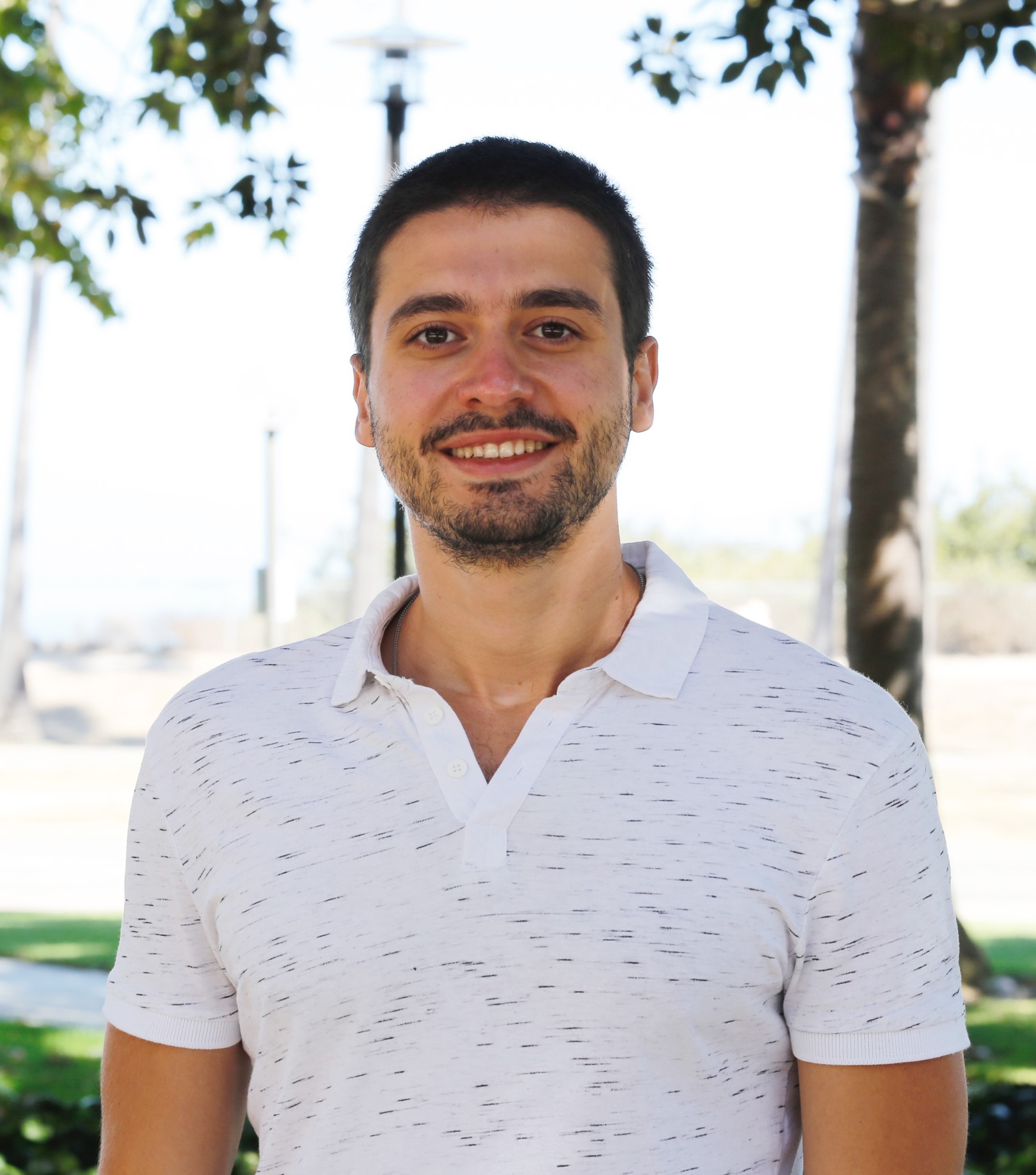}}]{BERKAY TURAN}is pursuing the Ph.D. degree in Electrical and Computer Engineering at the University of California, Santa Barbara. He received the B.Sc. degree in Electrical and Electronics Engineering as well as the B.Sc. degree in  Physics degree from \ Bo\u gazi\c ci University, Istanbul, Turkey, in 2018. His research interests include optimization and reinforcement learning for the design, control, and analysis of smart infrastructure systems such as the power grid and transportation systems.
\end{IEEEbiography}
\vskip -2\baselineskip plus -1fil
\begin{IEEEbiography}[{\includegraphics[width=1in,height=1.21in,clip,keepaspectratio]{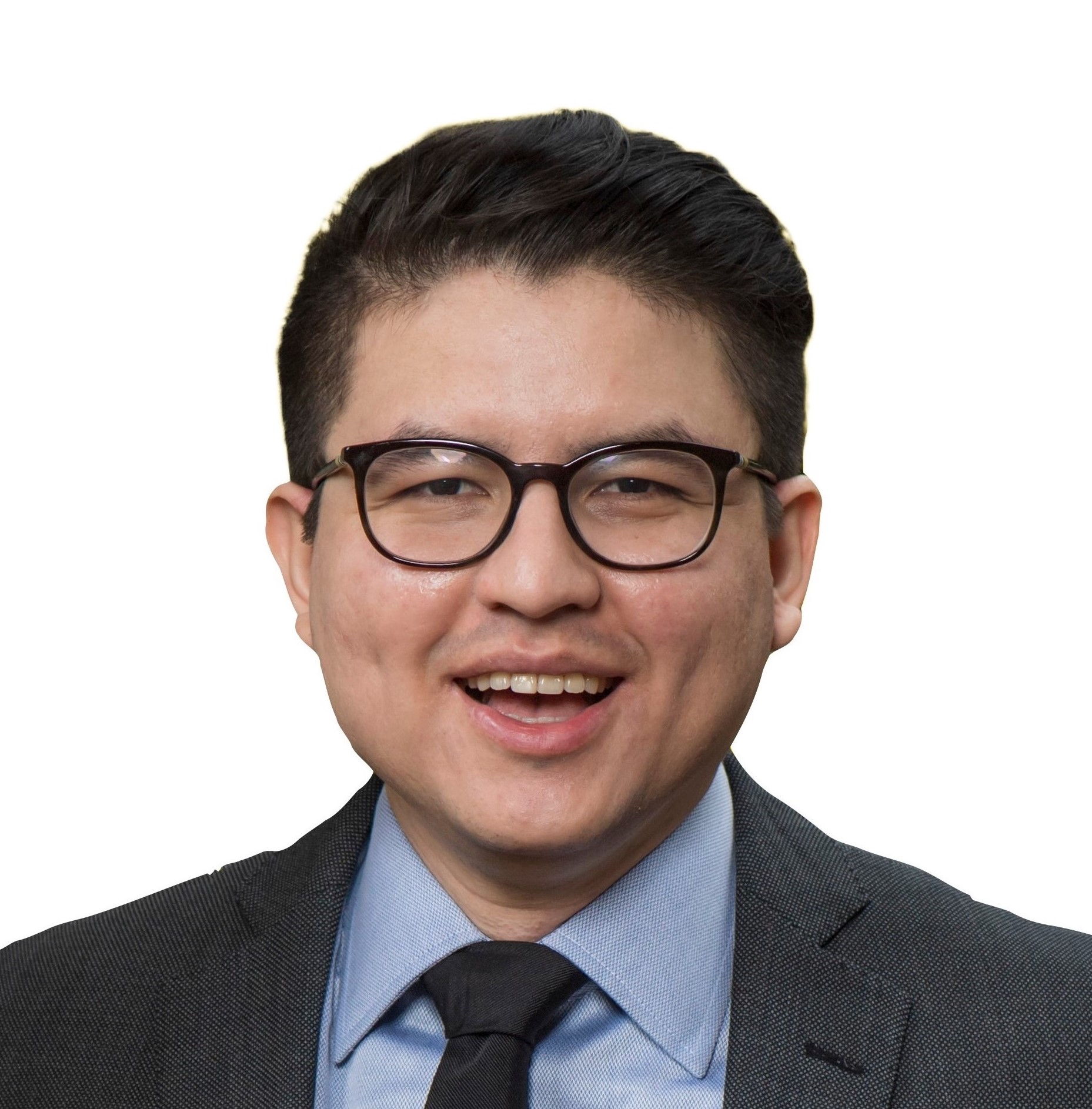}}]{C\'ESAR A. URIBE} is the Louis Owen Jr. Assistant Professor at the Department of Electrical and Computer Engineering at Rice University. He received the M.Sc. degrees in systems and control from Delft University of Technology, in The Netherlands, and in applied mathematics from the University of Illinois at Urbana-Champaign, in 2013 and 2016, respectively. He also received the PhD degree in electrical and computer engineering at the University of Illinois at Urbana-Champaign in 2018.  He was a Postdoctoral Associate in the Laboratory for Information and Decision Systems-LIDS at the Massachusetts Institute of Technology-MIT until 2020 and holds a visiting professor position at the Moscow Institute of Physics and Technology. His research interests include distributed learning and optimization, decentralized control, algorithm analysis, and computational optimal transport.
\end{IEEEbiography}
\vskip -2\baselineskip plus -1fil
\begin{IEEEbiography}[{\includegraphics[width=1in,height=1.21in,clip,keepaspectratio]{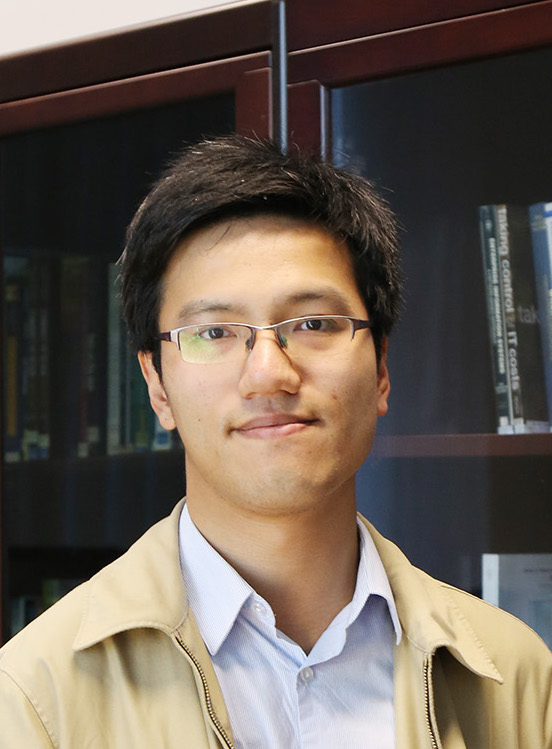}}]{HOI-TO WAI (S’11–M’18)} received his PhD degree from Arizona State University (ASU) in Electrical Engineering in Fall 2017, B. Eng. (with First Class Honor) and M. Phil. degrees in Electronic Engineering from The Chinese University of Hong Kong (CUHK) in 2010 and 2012, respectively. He is an Assistant Professor in the Department of Systems Engineering \& Engineering Management at CUHK. He has held research positions at ASU, UC Davis, Telecom ParisTech, Ecole Polytechnique, LIDS, MIT.
Hoi-To's research interests are in the broad area of signal processing, machine learning and distributed optimization, with a focus of their applications to network science. His dissertation has received the 2017's Dean's Dissertation Award from the Ira A. Fulton Schools of Engineering of ASU and he is a recipient of a Best Student Paper Award at ICASSP 2018.
\end{IEEEbiography}
\vskip -2\baselineskip plus -1fil
\begin{IEEEbiography}[{\includegraphics[width=1in,height=1.21in,clip,keepaspectratio]{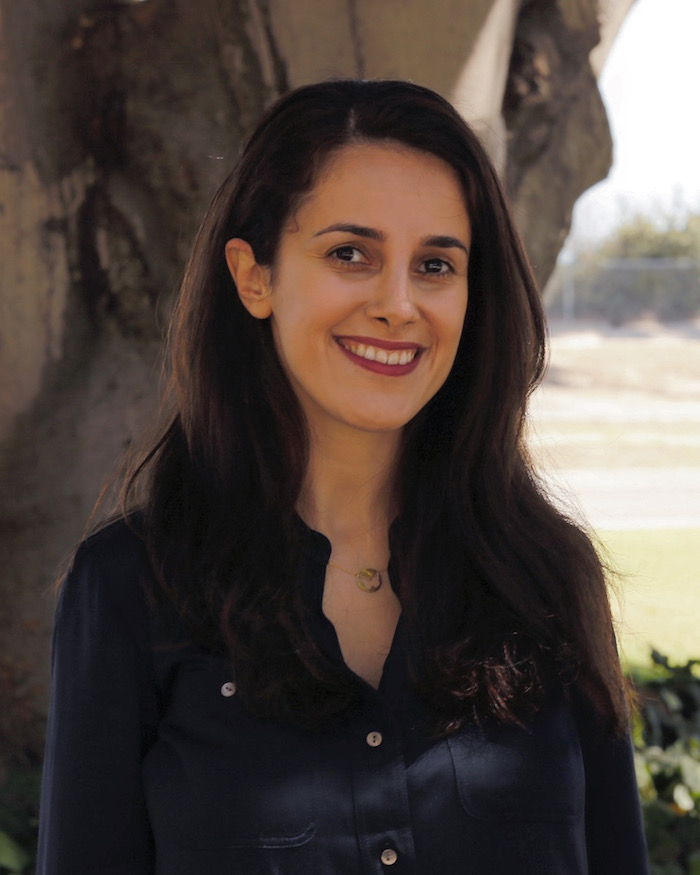}}]{MAHNOOSH ALIZADEH} is an assistant professor of Electrical and Computer Engineering at the University of California Santa Barbara. 
She received the B.Sc. degree (’09) in Electrical Engineering from Sharif University of Technology and the M.Sc. (’13) and Ph.D. (’14) degrees in Electrical and Computer Engineering from the University of California Davis. From 2014 to 2016, she was a postdoctoral scholar at Stanford University. Her research interests are focused on designing network control, optimization, and learning frameworks to promote efficiency and resiliency in societal-scale cyber-physical systems. Dr. Alizadeh is a recipient of the NSF CAREER award.
\end{IEEEbiography}

\newpage

\appendix
\subsection{Proof of Theorem~\ref{thm:stronglycvx}}\label{app:stronglycvx}
The proof will use the following theorem in \cite{boucheron2013concentration} as an auxiliary result, which offers a convenient bound for the expected value of the maximum of finitely many exponentially integrable random variables.

\begin{theorem}\cite[Theorem 2.5]{boucheron2013concentration}\label{thm:expectedmaximum}
Let $Z_1,\cdots,Z_N$ be real-valued random variables such that for every $\lambda\in(0,a)$ and $i=1,\dots,N$, the logarithm of the moment generating function of $Z_i$ satisfies $\mathbb{E}[e^{\lambda Z_i}]\leq \phi(\lambda)$ where $\phi$ is a convex and continuously differentiable function on $[0,a)$ with $0<a\leq\infty$ such that $\phi(0)=\phi'(0)=0$. Then
\begin{equation}
    \mathbb{E}[\underset{i=1,\dots,N}{\max}Z_i]\leq\underset{\lambda\in(0,a)}{\inf} \left[\frac{\log{N}+\phi(\lambda)}{
    \lambda}\right].
\end{equation}
\end{theorem}

Using the update rule, we write
\begin{align}
        &\|x_{t+1}-x^\star\|^2=\|\Pi_{\cal X}\{x_t-\gamma {\hat{\hat{g}}_t}/{\|\hat{\hat{g}}_t\|}\}-x^\star\|^2\\
        &\hspace{2cm}\leq \|x_t-x^\star-\gamma {\hat{\hat{g}}_t}/{\|\hat{\hat{g}}_t\|}\|^2\\
    &=\|x_t-x^\star\|^2-2\gamma\langle {\hat{\hat{g}}_t}/{\|\hat{\hat{g}}_t\|},x_t-x^\star\rangle+\gamma^2\\
    \begin{split}
    &=\|x_t-x^\star\|^2-2\gamma(1-Z_t)\langle\frac{\hat{\hat{g}}^{Z_t=0}_t}{\|\hat{\hat{g}}^{Z_t=0}_t\|},x_t-x^\star\rangle \\
    &\hspace{.5cm}-2\gamma Z_t\langle\frac{\hat{\hat{g}}^{Z_t=1}_t}{\|\hat{\hat{g}}^{Z_t=1}_t\|},x_t-x^\star\rangle+\gamma^2
    \end{split}\\
    \begin{split}
        &\leq\|x_t-x^\star\|^2-2\gamma(1-Z_t)\langle\frac{\nabla F(x_t)}{\|\nabla F(x_t)\|},x_t-x^\star\rangle \\
        &\hspace{.5cm}-2\gamma(1-Z_t)\langle\frac{\hat{\hat{g}}^{Z_t=0}_t}{\|\hat{\hat{g}}^{Z_t=0}_t\|}-\frac{\nabla F(x_t)}{\|\nabla F(x_t)\|},x_t-x^\star\rangle\\
    &\hspace{.5cm}+2\gamma Z_t\|x_t-x^\star\|+\gamma^2
    \end{split}\\
\begin{split}
            &\leq\|x_t-x^\star\|^2-2\gamma(1-Z_t)\frac{\mu}{L}\|x_t-x^\star\| \\
        &\hspace{.5cm}-2\gamma(1-Z_t)\langle\frac{\hat{\hat{g}}^{Z_t=0}_t}{\|\hat{\hat{g}}^{Z_t=0}_t\|}-\frac{\nabla F(x_t)}{\|\nabla F(x_t)\|},x_t-x^\star\rangle\\
    &\hspace{.5cm}+2\gamma Z_t\|x_t-x^\star\|+\gamma^2
    \end{split}\\
    \begin{split}\label{eq:thmstrongcvxcheckpt1}
            &\leq\|x_t-x^\star\|^2-\frac{2\gamma}{\kappa}(1-Z_t(1+\kappa))\|x_t-x^\star\| \\
        &\hspace{.5cm}+2\gamma\|e_t\|\|x_t-x^\star\|+\gamma^2,
    \end{split}
\end{align}
where 
\begin{equation}
    e_t=\frac{\nabla F(x_t)}{\|\nabla F(x_t)\|}-\frac{\hat{\hat{g}}^{Z_t=0}_t}{\|\hat{\hat{g}}^{Z_t=0}_t\|}
\end{equation}
In case of the event $Z_t=0$, we know that there are at least $(1-\alpha_2)N$ agents for which $Y_{i,t}=0$. Therefore, we define the following sets:
\begin{itemize}[wide, labelindent=0pt,topsep=.5mm]
    \item Define ${\cal T}^t$ as the set of $(1-\alpha_2)N$ agents with smallest indices $i\in[N]$ for which $Y_{i,t}=0$, i.e., 
    \begin{equation}
        {\cal T}^t=\{i|i\in[N], Y_{i,t}=0, \sum_{i}1=(1-\alpha_2)N\}
    \end{equation}
    such that $\sum_{i\in {\cal T}^t}i$ is minimized.
    \item For all agents $i\in{\cal T}^t$, define ${\cal T}^t_i$ as the set of $(1-\alpha_1)m$ smallest time indices $\tau\in[t-m+1,t]$ for which $W_{i,\tau}=0$, i.e.,
    \begin{equation}
        {\cal T}^t_i=\{\tau|\tau\in[t{-}m{+}1,t], W_{i,\tau}{=}0, \sum_\tau 1{=}(1{-}\alpha_1)m\}
    \end{equation}
    such that ${\sum_{\tau\in{\cal T}^t_i}}\tau$ is minimized.
\end{itemize}
Using the above sets, the following Lemma, whose proof can be found in Appendix~\ref{app:errorbound}, bounds the norm of $e_t$:
\begin{lemma}\label{lem:errorbound}
Suppose that $F(\cdot)$ and $F_{i,\tau}(\cdot)$, $\forall i,\tau$, are $L$-smooth. Define $\hat{\hat{g}}_t^{Z_t=0}$ as the robustified gradient at iteration $t$ when $Z_{t}=0$. Define the error $e_{t}$ as
$$    e_t=\frac{\nabla F(x_t)}{\|\nabla F(x_t)\|}-\frac{\hat{\hat{g}}^{Z_t=0}_t}{\|\hat{\hat{g}}^{Z_t=0}_t\|}$$
Then for all $t\geq m$, the following bounds the norm of the error:
\begin{equation}\label{eq:errorbound}
    \begin{split}
        &\|e_t\|\leq \frac{2}{\|\nabla F(x_t)\|}\Bigg(L\gamma(m{-}1)(1{+}C_{\alpha_1}{+}2C_{\alpha_2}(C_{\alpha_1}{+}1))\hspace{-.5cm}\\
        &+\left\|\nabla F(x_t){-}\frac{1}{(1{-}\alpha_1)m(1{-}\alpha_2)N}\sum_{i\in {\cal T}^t}\sum_{\tau\in {\cal T}_i^t}\nabla F_{i,\tau}(x_t)\right\|\hspace{-1cm}\\
        &+2C_{\alpha_2}\underset{i\in{\cal T}^t}{\max}\left\|\frac{1}{(1{-}\alpha_1)m}\sum_{\tau\in{\cal T}^t_i}\nabla F_{i,\tau}(x_t){-}\nabla F(x_t)\right\|_\infty\hspace{-.5cm}\\
        &+\frac{1}{(1{-}\alpha_2)N}\sum_{i\in {\cal T}^t}\underset{\tau,\tau'\in{\cal T}^t_i}{\max} C_{\alpha_1}\|\nabla F_{i,\tau}(x_{\tau'}){-}\nabla F_{i,\tau'}(x_{\tau'})\|_\infty\hspace{-1cm}\\
        &+2C_{\alpha_2}C_{\alpha_1}\underset{i\in{\cal T}^t,\tau,\tau'\in{\cal T}^t_i}{\max}\|\nabla F_{i,\tau}(x_{\tau'}){-}\nabla F_{i,\tau'}(x_{\tau'})\|_\infty
        \Bigg)\hspace{-1cm}
    \end{split}
\end{equation}
\end{lemma}
We plug \eqref{eq:errorbound} into \eqref{eq:thmstrongcvxcheckpt1} and take expectation of both sides with respect to $z\sim {\cal D}$ noting that $\nabla F_{i,\tau}(x_t)=\nabla F_{i,\tau'}(x_t)$, $\forall t,\tau,\tau'$ for the SAA, and use $\mu$-strong convexity of $F(\cdot)$:
\begin{align}
    \begin{split}
        &\underset{z\sim{\cal D}}{\mathbb{E}}[\|x_{t+1}-x^\star\|^2]\leq \underset{z\sim{\cal D}}{\mathbb{E}}[\|x_t-x^\star\|^2]+\gamma^2\\
        &-\frac{2\gamma}{\kappa}(1-Z_t(1+\kappa))\underset{z\sim{\cal D}}{\mathbb{E}}[\|x_t-x^\star\|]\\
        &+\frac{4\gamma^2L(m-1)(1+C_{\alpha_1}+2C_{\alpha_2}(C_{\alpha_1}+1))}{\mu}\\
        &+\frac{4\gamma\sigma}{\mu\sqrt{(1-\alpha_2)Nb}}\\
        &+\frac{8\gamma C_{\alpha_2}}{\mu}\underset{z\sim{\cal D}}{\mathbb{E}}[\underset{i\in{\cal T}^t}{\max}\|\nabla F_{i,t}(x_t)-\nabla F(x_t)\|_{\infty}]
    \end{split}\\
    \begin{split}\label{eq:thmstrongcvxcheckpt2}
        &\leq \underset{z\sim{\cal D}}{\mathbb{E}}[\|x_t-x^\star\|^2]+\gamma^2\\
        &-\frac{2\gamma}{\kappa}(1-Z_t(1+\kappa))\underset{z\sim{\cal D}}{\mathbb{E}}[\|x_t-x^\star\|]\\
        &+\frac{4\gamma^2L(m-1)(1+C_{\alpha_1}+2C_{\alpha_2}(C_{\alpha_1}+1))}{\mu}\\
        &+\frac{4\gamma\sigma}{\mu\sqrt{(1-\alpha_2)Nb}}\\
        &+\frac{8\gamma C_{\alpha_2}}{\mu}\underset{\lambda\in(0,b/a)}{\inf} \left[\frac{\log{(2(1-\alpha_2)Nd)}+b\phi(\lambda/b)}{
    \lambda}\right],\hspace{-2cm}
    \end{split}
\end{align}
where the second inequality uses Assumption~\ref{ass:subgamma} and Theorem~\ref{thm:expectedmaximum} with
\begin{align}\label{eq:35}
&\underset{z\sim{\cal D}}{\mathbb{E}}[\underset{i\in{\cal T}^t}{\max}\|\nabla F_{i,t}(x_t)-\nabla F(x_t)\|_{\infty}]
\nonumber\\
&= \underset{z\sim{\cal D}}{\mathbb{E}}[\underset{i\in{\cal T}^t,k\in[d]}{\max}|[\nabla F_{i,t}(x_t)-\nabla F(x_t)]_k|\\
&=\underset{z\sim{\cal D}}{\mathbb{E}}[\underset{i\in{\cal T}^t,k\in[d]}{\max}\big\{[\nabla F_{i,t}(x_t)-\nabla F(x_t)]_k,\nonumber\\
&\hspace{2.75cm}[\nabla F(x_t)-\nabla F_{i,t}(x_t)]_k\big\}
\end{align}
and
\begin{equation}
    \underset{z\sim D}{\mathbb{E}}[e^{\lambda [\nabla F_{i,t}(x_t)-\nabla F(x_t)]_k}]\leq b\phi(\lambda/b), 
\end{equation}
$\forall x_t\in{\cal X},k\in[d],|\lambda|\leq b/a,$ with $\phi(\lambda)=\frac{\lambda^2\sigma^2}{2(1-a|\lambda|)}$.
Note that the maximization is taken over $2|{\cal T}^t|d=2(1-\alpha_2)Nd$ sub-gamma random variables. The $\underset{\lambda\in(0,b/a)}{\inf}[\cdot]$ term attains its minimum at
\begin{equation}
    \lambda^\star=\frac{2\sqrt{\log{(2(1-\alpha_2)Nd)}}}{2a\sqrt{\log{(2(1-\alpha_2)Nd)}}/b+\sqrt{2\sigma^2/b}}
\end{equation}
and the expression evaluated at $\lambda^\star$ becomes
\begin{equation}\label{eq:39}
\begin{split}
        &\left[\frac{\log{(2(1-\alpha_2)Nd)}+b\phi(\lambda/b)}{
    \lambda}\right]_{\lambda=\lambda^\star}\\
    &=\frac{a}{b}\log{(2(1-\alpha_2)Nd)}+\frac{\sigma\sqrt{2\log{(2(1-\alpha_2)Nd)}}}{\sqrt{b}}
\end{split}
\end{equation}
The next step is to take expectation of \eqref{eq:thmstrongcvxcheckpt2} with respect to all randomness, where the challenge is to compute $\mathbb{E}[Z_t\|x_t-x^\star\|]$. However, $Z_t$ is a random variable that depends on $\{Y_{i,t}\}_{i\in[N]}$, and $Y_{i,t}$ is a random variable which depends on $\{W_{i,\tau}\}_{\tau\in[t-m+1,t]}$. All in all, $Z_t$ is a random variable which depends on the agents' state at time $t-m+1$ (which also not independent of history). Since $x_t$ is also dependent on the agents' state at time $t-m+1$, $Z_t$ and $x_t$ are dependent random variables. Therefore, we can not directly compute $\mathbb{E}[Z_t\|x_t-x^\star\|]$. But considering the fact that the two-state Markov chain governing the agents' states converges exponentially fast to its stationary distribution, the idea is to use total law of expectation by conditioning on the state at time $t-m+1-m_0$ for some $m_0\geq 0$, i.e.,
\begin{equation}\label{eq:trickstrongcvx}
    \mathbb{E}[Z_t\|x_t-x^\star\|]=\mathbb{E}[\mathbb{E}[Z_t\|x_t-x^\star\||{\cal S}_{t-m+1-m_0}]],
\end{equation}
where ${\cal S}_{t-m+1-m_0}=\{x_{t-m+1-m_0},\{\pi_{t-m+1-m_0}^i\}_{i\in[N]}\}$ and $\pi_{t-m+1-m_0}^i$ is the distribution of the state of agent $i$ at time $t-m+1-m_0$. Note that due to normalized updates:
\begin{equation}
    \|x_t-x^\star\|\leq \|x_{t-m+1-m_0}-x^\star\|+\gamma(m-1+m_0),
\end{equation}
and therefore \eqref{eq:trickstrongcvx} can be rewritten as
\begin{equation}
\begin{split}
        &\mathbb{E}[\mathbb{E}[Z_t\|x_t-x^\star\||{\cal S}_{t-m+1-m_0}]]\\
        &\leq
        \mathbb{E}[\| x_{t-m+1-m_0}-x^\star\|\mathbb{E}[Z_t|{\cal S}_{t-m+1-m_0}]]\\
        &\hspace{.5cm}+\mathbb{E}[\gamma(m-1+m_0)\mathbb{E}[Z_t|{\cal S}_{t-m+1-m_0}]].\label{eq:thmstrongcvxcheckpt3}
\end{split}
\end{equation}
We now use Lemma~\ref{lem:practicalbound} (or Lemma~\ref{lem:markovianprobbound} in Appendix~\ref{app:tighterbound}) to establish uniform bounds on $\mathbb{E}[Z_t|{\cal S}_{t-m+1-m_0}]$:
\begin{align}
\begin{split}
        \mathbb{E}[\|x_t-x^\star\|Z_t]&\leq \mathbb{E}[\|x_{t-m+1-m_0}-x^\star\|]P_Z^m(m_0)\\
        &\hspace{1cm}+\gamma(m-1+m_0)P_Z^m(m_0)
\end{split}\\
&\hspace{-2cm}\leq P_Z^m(m_0)\left(\mathbb{E}[\|x_t-x^\star\|]+2\gamma(m-1+m_0)\right),\label{eq:44}
\end{align}
where the last inequality follows from the normalized updates. Now we take expectation of \eqref{eq:thmstrongcvxcheckpt2} with respect to all randomness and use \eqref{eq:44}:
\begin{align}
    \begin{split}\label{eq:thmstrongcvxcheckpt4}
        \mathbb{E}&[\|x_{t+1}-x^\star\|^2]\leq \mathbb{E}[\|x_t-x^\star\|^2]+\gamma^2\\
        &-\frac{2\gamma}{\kappa}(1-P_Z^m(m_0)(1+\kappa))\mathbb{E}[\|x_t-x^\star\|]\\
        &+{4\gamma^2P_Z^m(m_0)(1+1/\kappa)(m-1+m_0)}\\
        &+{4\gamma^2\kappa(m-1)(1+C_{\alpha_1}+2C_{\alpha_2}(C_{\alpha_1}+1))}\\
        &+\frac{4\gamma \sigma}{\mu\sqrt{(1-\alpha_2)Nb}}+\frac{8\gamma C_{\alpha_2}}{\mu}\frac{a}{b}\log{(2(1-\alpha_2)Nd)}\\
        &+\frac{8\gamma C_{\alpha_2}}{\mu}\frac{\sigma\sqrt{2\log{(2(1-\alpha_2)Nd)}}}{\sqrt{b}}.
    \end{split}
\end{align}

Since $P_Z^m(m_0)<1/(1+\kappa)$, the coefficient of $\mathbb{E}[\|x_t-x^\star\|]$ term is negative. Therefore, to upper bound the inequality, we lower bound $\mathbb{E}\left[\|x_t-x^\star\|\right]$ as:
\begin{equation}
\begin{split}
    \mathbb{E}\left[\|x_t-x^\star\|\right]&=\mathbb{E}\left[\frac{\|x_t-x^\star\|^2}{\|x_t-x^\star\|}\right]\geq \frac{\mathbb{E}\left[\|x_t-x^\star\|^2\right]}{R}.
\end{split}
\end{equation}
Using above, we rewrite \eqref{eq:thmstrongcvxcheckpt4} for all $m_0\in {\mathbb N}_0$ such that $P_Z^m(m_0)<1/(1+\kappa)$:
\begin{align}\label{eq:thmstrongcvxcheckpt5}
    \begin{split}
        &\mathbb{E}[\|x_{t+1}-x^\star\|^2]\\
        &\leq \mathbb{E}[\|x_t-x^\star\|^2]\left(1-\frac{2\gamma}{\kappa R}(1-P_Z^m(m_0)(1+\kappa))\right)\\
        &+\gamma^2\overline{C}(m_0) +\frac{4\gamma \sigma}{\mu\sqrt{(1-\alpha_2)Nb}}\\
        &+\frac{8\gamma C_{\alpha_2}}{\mu}\frac{a}{b}\log{(2(1-\alpha_2)Nd)}\\
        &+\frac{8\gamma C_{\alpha_2}}{\mu}\frac{\sigma\sqrt{2\log{(2(1-\alpha_2)Nd)}}}{\sqrt{b}},
    \end{split}
\end{align}
where
\begin{equation}
\begin{split}
    \overline{C}(m_0)=&1+4P_Z^m(m_0)(1+1/\kappa)(m-1+m_0)\\
    &+4\kappa(m-1)(1+C_{\alpha_1}+2C_{\alpha_2}(C_{\alpha_1}+1)),
\end{split}
\end{equation}
\begin{equation}
    c_0(m_0)=\frac{2}{\kappa R}\left(1-P_Z^m(m_0)\left(1+\kappa\right)\right).
\end{equation}
Note that \eqref{eq:thmstrongcvxcheckpt5} holds for all $t\geq m+m_0$. Hence, we can write:
\begin{align}
    \begin{split}
        &\mathbb{E}\left[\|x_{T+m+m_0}{-}x^\star\|^2\right]\leq \mathbb{E}\left[\|x_{m+m_0}{-}x^\star\|^2\right](1{-}c_0(m_0)\gamma)^T\hspace{-1cm}\\
        &+\frac{4\gamma \sigma}{\mu\sqrt{(1-\alpha_2)Nb}}\sum_{t=m+m_0}^{T+m+m_0-1}\prod_{i=t+1}^{T+m+m_0-1}(1{-}c_0(m_0)\gamma)\hspace{-1cm}\\
        &+\frac{8\gamma C_{\alpha_2}}{\mu}\sum_{t=m+m_0}^{T+m+m_0-1}\prod_{i=t+1}^{T+m+m_0-1}(1-c_0(m_0)\gamma)\times\\
        &\Bigg(\frac{a}{b}\log{(2(1-\alpha_2)Nd)}+\frac{\sigma\sqrt{2\log{(2(1-\alpha_2)Nd)}}}{\sqrt{b}}\Bigg)\\
       &+\overline{C}(m_0)\gamma^2\sum_{t=m+m_0}^{T+m+m_0-1}\prod_{i=t+1}^{T+m+m_0-1}(1-c_0(m_0)\gamma)
    \end{split}\\
    \begin{split}
&=\mathbb{E}\left[\|x_{m+m_0}-x^\star\|^2\right](1-c_0(m_0)\gamma)^T\\
        &+\frac{4 \sigma}{\mu\sqrt{(1-\alpha_2)Nb}}\frac{(1-(1-c_0(m_0)\gamma)^T)}{c_0(m_0)}\\
         &+\frac{8 C_{\alpha_2}}{\mu}\frac{a}{b}\log{(2(1-\alpha_2)Nd)}\frac{(1-(1-c_0(m_0)\gamma)^T)}{c_0(m_0)}\\
        &+\frac{8 C_{\alpha_2}}{\mu}\frac{\sigma\sqrt{2\log{(2(1-\alpha_2)Nd)}}}{\sqrt{b}}\frac{(1-(1-c_0(m_0)\gamma)^T)}{c_0(m_0)}\hspace{-1cm}\\
       &+\overline{C}(m_0)\gamma\frac{(1-(1-c_0(m_0)\gamma)^T)}{c_0(m_0)}.
    \end{split}
\end{align}
Next, we observe that $(1-(1-c_0(m_0)\gamma)^T)\leq 1$ and $\|x_{m+m_0}-x^\star\|\leq (\|x_1-x^\star\|+(m+m_0-1)\gamma)^2$ to get the desired result:
\begin{equation}
    \begin{split}
        &\mathbb{E}[\|x_{T+m+m_0}-x^\star\|^2]\leq\\
        &\left(\|x_{1}-x^\star\|+\gamma(m+m_0-1)\right)^2(1-c_0(m_0)\gamma)^T\\
        &+\frac{4\sigma}{\mu\sqrt{(1-\alpha_2)Nb}c_0(m_0)}+\frac{\overline{C}(m_0)\gamma}{c_0(m_0)}\\
        &+\frac{8a C_{\alpha_2}\log{(2(1{-}\alpha_2)Nd)}}{\mu c_0(m_0) b}{+}\frac{8 C_{\alpha_2}\sigma\sqrt{2\log{(2(1{-}\alpha_2)Nd)}}}{\mu c_0(m_0)\sqrt{b}}.\hspace{-1cm}
    \end{split}
\end{equation}

\subsection{Proof of Lemma~\ref{lem:errorbound}}\label{app:errorbound}
\begin{align}
    \|e_t\|&=\left\|\frac{\nabla F(x_t)}{\|\nabla F(x_t)\|}-\frac{\hat{\hat{g}}^{Z_t=0}_t}{\|\hat{\hat{g}}^{Z_t=0}_t\|}\right\|\\
    &=\left\|\frac{\nabla F(x_t)\|\hat{\hat{g}}^{Z_t=0}_t\|-\hat{\hat{g}}^{Z_t=0}_t\|\nabla F(x_t)\|}{\|\nabla F(x_t)\|\|\hat{\hat{g}}^{Z_t=0}_t\|}\right\|\\
    \begin{split}
        &=\Bigg\|\frac{\|\hat{\hat{g}}^{Z_t=0}_t\|(\nabla F(x_t)-\hat{\hat{g}}^{Z_t=0}_t)}{\|\nabla F(x_t)\|\|\hat{\hat{g}}^{Z_t=0}_t\|}\\
        &\hspace{1cm}+\frac{\hat{\hat{g}}^{Z_t=0}_t(\|\hat{\hat{g}}^{Z_t=0}_t\|-\|\nabla F(x_t)\|)}{\|\nabla F(x_t)\|\|\hat{\hat{g}}^{Z_t=0}_t\|}\Bigg\|
    \end{split}\\
    &\leq 2\frac{\|\nabla\label{eq:errorbound1} F(x_t)-\hat{\hat{g}}^{Z_t=0}_t\|}{\|\nabla F(x_t)\|}
\end{align}

Using the sets ${\cal T}^t$ and $\{{\cal T}_i^t\}_{i\in{\cal T}^t}$, we define:
\begin{align}
    \Bar{\Bar{g}}_t&=\frac{1}{(1-\alpha_1)m(1-\alpha_2)N}\sum_{i\in {\cal T}^t}\sum_{\tau\in {\cal T}_i^t}\nabla F_{i,\tau}(x_\tau),\\
    \Bar{\hat{g}}_t&=\frac{1}{(1-\alpha_2)N}\sum_{i\in{\cal T}^t}\hat{g}_{i,t}.
\end{align}
Here, $\Bar{\Bar{g}}$ is the true mean of $(1-\alpha_1)m(1-\alpha_2)N$ trustworthy gradients with time indices $\tau\in{\cal T}^i_t$ from agents $i\in{\cal T}^t$, and $\Bar{\hat{g}}_t$ is the mean of the robustified gradients of agents $i\in{\cal T}^t$.

We split the numerator of \eqref{eq:errorbound1} as follows:
\begin{equation}
    \|\nabla F(x_t)-\hat{\hat{g}}^{Z_t=0}_t\|\leq \|\nabla F(x_t)-\Bar{\Bar{g}}_t\|+\|\Bar{\Bar{g}}_t-\Bar{\hat{g}}_t\|+\|\Bar{\hat{g}}_t-\hat{\hat{g}}^{Z_t=0}_t\|.
\end{equation}
Next, we upper bound each term above using smoothness of $F$, normalized updates, and triangular inequality:
\begin{enumerate}[wide, labelindent=0pt,topsep=.5mm]
\item $\|\nabla F(x_t)-\Bar{\Bar{g}}_t\|$
\begin{align}
    &{=}\|\nabla F(x_t){-}\frac{1}{(1{-}\alpha_1)m(1{-}\alpha_2)N}\hspace{-.1cm}\sum_{i\in {\cal T}^t}\hspace{-.1cm}\sum_{\tau\in {\cal T}_i^t}\hspace{-.1cm}\nabla F_{i,\tau}(x_\tau)\|\\
    \begin{split}
        &\leq \|\nabla F(x_t){-}\frac{1}{(1{-}\alpha_1)m(1{-}\alpha_2)N}\hspace{-.1cm}\sum_{i\in {\cal T}^t}\hspace{-.1cm}\sum_{\tau\in {\cal T}_i^t}\hspace{-.1cm}\nabla F_{i,\tau}(x_t)\|\hspace{-1cm}\\
        &+\|\frac{1}{(1{-}\alpha_1)m(1{-}\alpha_2)N}\hspace{-.1cm}\sum_{i\in {\cal T}^t}\hspace{-.1cm}\sum_{\tau\in {\cal T}_i^t}\hspace{-.1cm}\nabla F_{i,\tau}(x_t){-}\nabla F_{i,\tau}(x_\tau)\|\hspace{-1cm}
    \end{split}\\
       &\nonumber\leq \|\nabla F(x_t)-\frac{1}{(1-\alpha_1)m(1-\alpha_2)N}\sum_{i\in {\cal T}^t}\sum_{\tau\in {\cal T}_i^t}\nabla F_{i,\tau}(x_t)\|\\
        &+\frac{L}{(1-\alpha_1)m(1-\alpha_2)N}\sum_{i\in {\cal T}^t}\sum_{\tau\in {\cal T}_i^t}\|x_t-x_\tau\|\\
        &\nonumber\leq  \|\nabla F(x_t)-\frac{1}{(1-\alpha_1)m(1-\alpha_2)N}\sum_{i\in {\cal T}^t}\sum_{\tau\in {\cal T}_i^t}\nabla F_{i,\tau}(x_t)\|\\
        &+L\gamma(m-1).
\end{align}
\item $\|\Bar{\Bar{g}}_t-\Bar{\hat{g}}_t\| =\|\frac{1}{(1-\alpha_1)N}\sum_{i\in{\cal T}^t}\Bar{g}_{i,t}-\hat{g}_{i,t}\|$
\begin{align}
    &\leq \frac{1}{(1-\alpha_2)N}\sum_{i\in{\cal T}^t}\|\Bar{g}_{i,t}-\hat{g}_{i,t}\|\\
    \begin{split}\label{eq:122}
     &\leq \frac{1}{(1-\alpha_2)N}\sum_{i\in{\cal T}^t}\underset{\tau \in {\cal T}_i^t}{\max}C_{\alpha_1}\|\nabla F_{i,\tau}(x_\tau)\\&
     \hspace{2.5cm}-\frac{1}{(1-\alpha_1)m}\sum_{\tau'\in {\cal T}_i^t}\nabla F_{i,\tau'}(x_{\tau'})\|_\infty
    \end{split}\\
    &{\leq} \frac{1}{(1{-}\alpha_2)N}\hspace{-.15cm}\sum_{i\in {\cal T}^t}\hspace{-.1cm}\underset{\tau,\tau'\in{\cal T}^t_i}{\max}\hspace{-.1cm} C_{\alpha_1}\hspace{-.05cm}\|\hspace{-.05cm}\nabla F_{i,\tau}(x_\tau){-}\nabla F_{i,\tau'}(x_{\tau'})\hspace{-.05cm}\|_\infty\hspace{-.1cm}\\
    \begin{split}
        &{\leq}  \frac{1}{(1{-}\alpha_2)N}\hspace{-.1cm}\sum_{i\in {\cal T}^t}\hspace{-.1cm}\underset{\tau,\tau'\in{\cal T}^t_i}{\max}\hspace{-.1cm} C_{\alpha_1}\|\nabla F_{i,\tau}(x_\tau){-}\nabla F_{i,\tau}(x_{\tau'})\|_\infty\hspace{-1cm}\\
        &+\frac{1}{(1{-}\alpha_2)N}\hspace{-.1cm}\sum_{i\in {\cal T}^t}\hspace{-.1cm}\underset{\tau,\tau'\in{\cal T}^t_i}{\max}\hspace{-.1cm} C_{\alpha_1}\|\nabla F_{i,\tau}(x_{\tau'}){-}\nabla F_{i,\tau'}(x_{\tau'})\|_\infty\hspace{-1cm}
    \end{split}\\
    \begin{split}\label{eq:126}
    &\leq C_{\alpha_1}L\gamma(m-1)\\
    &{+}\frac{1}{(1{-}\alpha_2)N}\hspace{-.1cm}\sum_{i\in {\cal T}^t}\hspace{-.1cm}\underset{\tau,\tau'\in{\cal T}^t_i}{\max}\hspace{-.1cm} C_{\alpha_1}\|\nabla F_{i,\tau}(x_{\tau'}){-}\nabla F_{i,\tau'}(x_{\tau'})\|_\infty.\hspace{-1cm}
    \end{split}
\end{align}
\item $\|\Bar{\hat{g}}_t-\hat{\hat{g}}^{Z_t=0}_t\|\leq C_{\alpha_2}\underset{i\in{\cal T}^t}{\max}\|\hat{g}_{i,t}-\frac{1}{(1-\alpha_2)N}\sum_{j\in{\cal T}^t}\hat{g}_{j,t}\|_\infty$
\begin{align}
    &\leq C_{\alpha_2}\underset{i,j\in{\cal T}^t}{\max}\|\hat{g}_{i,t}-\hat{g}_{j,t}\|_\infty\\
    &{\leq} C_{\alpha_2}\underset{i,j\in{\cal T}^t}{\max}(\|\hat{g}_{i,t}{-}\Bar{g}_{i,t}\|_\infty{+}\|\hat{g}_{j,t}{-}\Bar{g}_{j,t}\|_\infty{+}\|\Bar{g}_{i,t}{-}\Bar{g}_{j,t}\|_\infty)\\
    &\leq 2C_{\alpha_2}\underset{i\in{\cal T}^t}{\max}\|\hat{g}_{i,t}{-}\Bar{g}_{i,t}\|{+}2C_{\alpha_2}\underset{i\in{\cal T}^t}{\max}\|\Bar{g}_{i,t}{-}\nabla F(x_t)\|_\infty\\
        &\nonumber\leq 2C_{\alpha_2}\underset{i\in{\cal T}^t}{\max}\|\hat{g}_{i,t}-\Bar{g}_{i,t}\|\\
        &\nonumber+2C_{\alpha_2}\underset{i\in{\cal T}^t}{\max}\|\frac{1}{(1-\alpha_1)m}\sum_{\tau\in{\cal T}_i^t}\nabla F_{i,\tau}(x_\tau)-\nabla F_{i,\tau}(x_t)\|_\infty\\
        &+2C_{\alpha_2}\underset{i\in{\cal T}^t}{\max}\|\frac{1}{(1{-}\alpha_1)m}\sum_{\tau\in{\cal T}^t_i}\nabla F_{i,\tau}(x_t){-}\nabla F(x_t)\|_\infty\\
    \begin{split}
      & \leq 2C_{\alpha_2}\underset{i\in{\cal T}^t}{\max}\|\hat{g}_{i,t}-\Bar{g}_{i,t}\|+2C_{\alpha_2}L\gamma(m-1)\\
        &{+}2C_{\alpha_2}\underset{i\in{\cal T}^t}{\max}\|\frac{1}{(1{-}\alpha_1)m}\sum_{\tau\in{\cal T}^t_i}\nabla F_{i,\tau}(x_t){-}\nabla F(x_t)\|_\infty
    \end{split}\\
        &\nonumber\leq 2(C_{\alpha_1}+1)C_{\alpha_2}L\gamma(m-1)\\
        &\nonumber{+}2C_{\alpha_2}C_{\alpha_1}\underset{i\in{\cal T}^t,\tau,\tau'\in{\cal T}^t_i}{\max}\|\nabla F_{i,\tau}(x_{\tau'})-\nabla F_{i,\tau'}(x_{\tau'})\|_\infty\\
        &{+}2C_{\alpha_2}\underset{i\in{\cal T}^t}{\max}\|\frac{1}{(1{-}\alpha_1)m}\sum_{\tau\in{\cal T}^t_i}\nabla F_{i,\tau}(x_t){-}\nabla F(x_t)\|_\infty,
\end{align}
where the last inequality follows from \eqref{eq:122}-\eqref{eq:126}. Gathering all three terms, we get the desired result.
\end{enumerate}


\bt{
\subsection{Proof of Lemma~\ref{lem:practicalbound}}\label{app:practicalbound}
The goal is to find a uniform bound on $\mathbb{E}[Z_t|{\cal S}_{t-m+1-m_0}]$ independent of the system state at $t-m+1-m_0$. Note that
\begin{align}
    \mathbb{E}&[Z_t|{\cal S}_{t-m+1-m_0}]={\mathbb P}(Z_t=1|{\cal S}_{t-m+1-m_0})\\
    &={\mathbb P}\left(\textstyle\frac{1}{N}\sum_{i\in[N]}Y_{i,t}>\alpha_2|{\cal S}_{t-m+1-m_0}\right)\\
    &\hspace{-.5cm}\leq\sum_{k=\alpha_2 N+1}^N \binom{N}{k}(P_Y^m(m_0))^k(1-P_Y^m(m_0))^{(N-k)},
\end{align}
where
\begin{align}
        &P_Y^m(m_0){=}P_Y^m(m_0,m,\alpha_1,M){=}\underset{i\in[N],t}{\max}\mathbb{E}[Y_{i,t}|{\cal S}_{t{-}m{+}1{-}m_0}]\\
        &=\underset{i\in[N],t}{\max}\mathbb{P}[Y_{i,t}=1|{\cal S}_{t-m+1-m_0}]\\
        &=\underset{i\in[N],t}{\max}\mathbb{P}\left( \frac{1}{m}\sum_{\tau=t-m+1}^tW_{i,\tau}>\alpha_1|{\cal S}_{t-m+1-m_0}\right)\label{eq:PZcheckpoint1}
\end{align}
We can rewrite \eqref{eq:PZcheckpoint1} using total law of probability:
\begin{align}
      \nonumber P_Y^m(m_0){=}\underset{i\in[N]}{\max}\sum_{s=0}^1&\mathbb{P}\left( \frac{1}{m}\sum_{\tau=t-m+1}^t
       \hspace{-.4cm}W_{i,\tau}{>}\alpha_1|W_{i,t-m+1}=s\right)\\
       \label{eq:PY}&\times\mathbb{P}_{\pi^i_{t-m+1-m_0}}(W_{i,t-m+1}{=}s),
\end{align}
where $\mathbb{P}_{\pi^i_{t-m+1-m_0}}(W_{i,t-m+1}=s)$ is the probability that $W_{i,t-m+1}=s$ given the distribution at time $t-m+1-m_0$. Let
\begin{equation}
    P_s(m,\alpha_1m){=}\mathbb{P}\left(\sum_{\tau=t-m+1}^t\hspace{-.4cm}W_{i,\tau}{>}\alpha_1 m|W_{i,t-m+1}{=}s\right)
\end{equation}
Accordingly, we know that:
\begin{equation}
    P_0(m,\alpha_1m){=}p_bP_1(m{-}1,\alpha_1m){+}(1{-}p_b)P_0(m{-}1,\alpha_1m)
\end{equation}
Similarly,
\begin{align}
\begin{split}
    &P_1(m,\alpha_1m)=(1-p_t)P_1(m-1,\alpha_1m-1)\\
    &\hspace{2.5cm}+p_tP_0(m-1,\alpha_1m-1)
\end{split}\\
    &\geq (1-p_t)P_1(m-1,\alpha_1m)+p_tP_0(m-1,\alpha_1m)\\
    \begin{split}
    &=P_0(m,\alpha_1m)\\
    &{+}(1{-}p_t{-}p_b)(P_1(m{-}1,\alpha_1m){-}P_0(m{-}1,\alpha_1m))   
    \end{split}
\end{align}
Since $(1-p_t-p_b)\geq 0$, $P_1(m,\alpha_1m)\geq P_0(m,\alpha_1m)$ if $P_1(m-1,\alpha_1m)\geq P_0(m-1,\alpha_1m)$. We know that
\begin{equation}
    0=P_0(\alpha_1m+1,\alpha_1m)\leq P_1(\alpha_1m+1,\alpha_1m)
\end{equation}
because when $W_{i,t-m+1}=0$, then there can be at most $\alpha_1m$ instances where $W_{i,\tau}=1$ for $\tau\in[t-m+1,t-m+1+\alpha_1m]$. Therefore, the probability of having the sum strictly larger than $\alpha_1 m$ is zero. This establishes that $P_1(m,\alpha_1m)\geq P_0(m,\alpha_1m)$. We also have a closed form for
\begin{align}
\begin{split}
      &\mathbb{P}_{\pi^i_{t-m+1-m_0}}(W_{i,t-m+1}=1)=\\
      &+\pi^i_{t-m+1-m_0}(0)\frac{p_b-p_b(1-p_b-p_t)^{m_0}}{p_b+p_t}\\
      &+\pi^i_{t-m+1-m_0}(1)\frac{p_b+p_t(1-p_b-p_t)^{m_0}}{p_b+p_t}  
\end{split}\\
&\leq \frac{p_b+p_t(1-p_b-p_t)^{m_0}}{p_b+p_t},
\end{align}
where the inequality holds with equality if $\pi_{t-m+1-m_0}^i(1)=1$, i.e., at time $t-m+1-m_0$ the agent was corrupted, which is the worst-case intuition. At this point, we have:
\begin{align}
    P_Y^m(m_0)&{=}\underset{i\in[N]}{\max}\sum_{s=0}^1 P_s(m,\alpha_1m)\mathbb{P}_{\pi^i_{t{-}m{+}1{-}m_0}}(W_{i,t{-}m+{1}}{=}s)\\
    &\leq\sum_{s=0}^1 P_s(m,\alpha_1m) \Pi_{m_0}(i),\label{eq:intialdistcomparison}
\end{align}
where
\begin{align}
    &\Pi_{m_0}(0)=\frac{p_t-p_t(1-p_b-p_t)^{m_0}}{p_b+p_t}\\
    &\Pi_{m_0}(1)=\frac{p_b+p_t(1-p_b-p_t)^{m_0}}{p_b+p_t}.
\end{align}
We have established that the initial distribution given above maximizes $P_Y^m(m_0)$. Next, we argue that $P_Y^m(m_0)$ is larger over the Markov chain governed by $M^{m_0}$ with
\begin{equation}\label{eq:modifiedmarkovchain}
   M_{m_0}=\begin{bmatrix}
1-p_b'(m_0) & p_b'(m_0)\\
p_t'(m_0) & 1-p_t'(m_0)
\end{bmatrix},
\end{equation}
where $p_b'(m_0)=p_b+p_t(1-p_b-p_t)^{m_0}$ and $p_t'(m_0)=p_t-p_t(1-p_b-p_t)^{m_0}$. With some abuse of notation, let
\begin{equation}
    P_s^{M^m}(m,\alpha_1m){=}{P}_s^{M^m}\left(\sum_{\tau=t-m+1}^t\hspace{-.4cm}W_{i,\tau}{>}\alpha_1 m|W_{i,t-m+1}{=}s\right),
\end{equation}
where the superscript $M^m$ denotes that the random variable follows the Markov chain governed by $M$ for $m$ time steps. Using $P_1(\cdot)\geq P_0(\cdot)$, we have that
\begin{align}
\begin{split}
    &P_1^{M^m}(m,\alpha_1m){=}P_1^{M^{m{-}1}}(m{-}1,\alpha_1m{-}1)(1{-}p_t)\\
    &\hspace{2cm}+P_0^{M^{m-1}}(m-1,\alpha_1m-1)p_t
\end{split}\\
\begin{split}
    \leq&P_1^{M^{m-1}}(m-1,\alpha_1m-1)(1-p_t'(m_0))\\
    &+P_0^{M^{m-1}}(m-1,\alpha_1m-1)p_t'(m_0) 
\end{split}\\
=&P_1^{M_{m_0}M^{m-1}}(m,\alpha_1m),
\end{align}
where the superscript $M_{m_0}M^{m-1}$ indicates that the random variable follows the Markov chain governed by $M_{m_0}$ initially and $M$ for the next $m-1$ time steps. Similarly,
\begin{align}
    \begin{split}
    &P_0^{M^m}(m,\alpha_1m)=P_1^{M^{m-1}}(m-1,\alpha_1m)p_b\\
    &\hspace{2cm}+P_0^{M^{m-1}}(m-1,\alpha_1m)(1-p_b)
\end{split}\\
\begin{split}
    \leq&P_1^{M^{m-1}}(m-1,\alpha_1m-1)p_b'(m_0) \\
    &+P_0^{M^{m-1}}(m-1,\alpha_1m-1)(1-p_b'(m_0) )
\end{split}\\
=&P_0^{M_{m_0}M^{m-1}}(m,\alpha_1m).
\end{align}
Recursively applying this with the boundary conditions $P_s^{M^{\alpha_1m}}(\alpha_1m,\alpha_1m)=P_s^{M_{m_0}^{\alpha_1m}}(\alpha_1m,\alpha_1m)=0$ for $s=1,2$, we get:
\begin{equation}
    P_s^{M^m}(m,\alpha_1m)\leq P_s^{M^m_{m_0}}(m,\alpha_1m).\label{eq:newchaincomparison}
\end{equation}
Therefore, we can upper bound $P_Y^m(m_0)$ as
\begin{equation}
    P_Y^m(m_0)\leq\sum_{s=0}^1P_s^{M^m_{m_0}}(m,\alpha_1m)\Pi_{m_0}(s).
\end{equation}
The above is essentially the equal to the following:
\begin{equation}
    P_Y^m(m_0)\leq\mathbb{P}_{\Pi_{m_0}}(\frac{1}{m}\sum_{\tau=t-m+1}^t W_{i,\tau}>\alpha_1),
\end{equation}
where the subscript $\Pi_{m_0}$ denotes the initial distribution of $W_{i,\tau}$. Since the initial distribution of $W_{i,\tau}$ is equal to the stationary distribution of $M^{m_0}$, we can use the following theorem to bound $P_Y^m(m_0)$:
\begin{theorem}\cite[Theorem 1]{leon2004optimal}
For all pairs $((X_n), f )$, such that $(X_n)$ is a finite, ergodic and reversible Markov chain in stationary state with second largest eigenvalue $\lambda$ and $f$ is a function taking values in $[0, 1]$ such that $E[f (X_i)] = \mu$, the following bounds, with $\lambda_0 = max(0, \lambda)$, hold for all $\epsilon > 0$ such that $\mu + \epsilon < 1$ and all time n
\begin{equation}
    \mathbb{P}(\sum_{i=1}^nf(X_i)\geq n(\mu+\epsilon))\leq\exp{(-2\frac{1-\lambda_0}{1+\lambda_0}n\epsilon^2)}.
\end{equation}
\end{theorem}
Applying the above theorem with $\lambda=(1-p_b-p_t)$, $n=m$, $\mu+\epsilon=\alpha_1m$, $\mu=\Pi_{m_0}(1)$, and $1+\lambda_0=2-p_b-p_t\leq 2$:
\begin{equation}
    P_Y^m(m_0)\leq \exp{\left(-m(\alpha_1-\Pi_{m_0}(1))^2(p_b+p_t)\right)},
\end{equation}
gives the desired bound on $P_Y^m(m_0)$. Note that $P_Z^m(m_0)$ is equal to $1-F_B(\alpha_2N,N,P_Y^m(m_0))$, where $F_B(\alpha_2N,N,P_Y^m(m_0)$ is the cumulative distribution function of the binomial distribution with parameters $(N,P_Y^m(m_0))$ evaluated at $\alpha_2N$. $F_B(\alpha_2N,N,P_Y^m(m_0)$ is given by \cite{wadsworth1960introduction}:
\begin{equation}
    (N{-}\alpha_2N)\binom{N}{\alpha_2N}\int_{0}^{1{-}P_Y^m(m_0)}\hspace{-.4cm}t^{N{-}\alpha_2 N{-}1}(1{-}t)^{\alpha_2N}dt,
\end{equation}
which decreases with $P_Y^m(m_0)$. Therefore, $P_Z^m(m_0)=1-F_B(\alpha_2N,N,P_Y^m(m_0))$ is maximized when $P_Y^m(m_0)$ is maximized, which gives the desired bound on $P_Z^m(m_0)$.
}

\bt{
\subsection{Sensitivity Analysis of $P_Y^m(m_0)$ with respect to $p_b$ and $p_t$}\label{app:Pypbpt}
Let $u=(\alpha_1-\Pi_{m_0}^1)^2(p_b+p_t)$. By chain rule, we have
\begin{align}
    \frac{d P_Y^m(m_0)}{d p_i}&=\frac{d P_Y^m(m_0)}{d u}\times \frac{d u}{d p_i}\\
    &=-m e^{-mu} \frac{d u}{d p_i},~~i=b,t.
\end{align}
The derivative of $u$ with respect to $p_b$ is given by:
\begin{equation}
\begin{split}
   &\frac{d u}{d p_b}{=}\frac{\alpha_1{-}\Pi_{m_0}^1}{p_b{+}p_t}{\times}(p_b(\alpha_1{+}2m_0p_t(1{-}p_b{-}p_t)^{m_0{-}1}{-}1)\\
   &{+}p_t(\alpha_1{+}2m_0p_t(1{-}p_b{-}p_t)^{m_0{-}1}{+}(1{-}p_b{-}p_t)^{m_0}{-}2))
\end{split}
\end{equation}
Noting that $\alpha_1> \Pi_{m_0}^m$, $p_t\leq p_t+p_b$, and $(1-x)^{1/x}\leq e^{-1}$ for $x\in[0,1]$:
\begin{align}
    \nonumber&\frac{d u}{d p_b}{\leq} \frac{\alpha_1{-}\Pi_{m_0}^1}{p_b{+}p_t}\times(p_b(\alpha_1{+}2m_0(p_b{+}p_t)e^{{-}(m_0{-}1)(p_b{+}p_t)}{-}1)\\
    &{+}p_t(\alpha_1{+}2m_0(p_b{+}p_t)e^{{-}(m_0{-}1)(p_b{+}p_t)}+e^{{-}m_0(p_b{+}p_t)}{-}2))
\end{align}
Now, we set $m_0=k_0/(p_b+p_t)$ for some $k_0\geq 1$:
\begin{align}
    \nonumber\frac{d u}{d p_b}\leq& \frac{\alpha_1-\Pi_{m_0}^m}{p_b+p_t}\times(p_b(\alpha_1+2k_0e^{-k_0+p_b+p_t}-1)\\
    &+p_t(\alpha_1+2k_0e^{-k_0+p_b+p_t}+e^{-k_0}-2))\\
    \nonumber\leq &\frac{\alpha_1-\Pi_{m_0}^m}{p_b+p_t}\times(p_b(\alpha_1+2k_0e^{-k_0+1}-1)\\
    &+p_t(\alpha_1+2k_0e^{-k_0+1}+e^{-k_0}-2))\\
    &< 0,
\end{align}
where the last inequality holds when $k_0\geq 4$. Therefore, when $m_0\geq 4/(p_b+p_t)$, we have that $du/dp_b<0$, which implies that $dP_Y^m(m_0)/dp_b>0$. Similarly,
\begin{equation}
\begin{split}
   &\frac{d u}{d p_t}=\frac{\alpha_1-\Pi_{m_0}^1}{p_b+p_t}\times\\
   &(p_t(\alpha_1+(2m_0+p_b+p_t-1)(1-p_b-p_t)^{m_0-1})\\
   &{+}p_b(\alpha_1{+}2m_0p_t(1{-}p_b{-}p_t)^{m_0{-}1}{-}2(1{-}p_b{-}p_t)^{m_0}{+}1))\hspace{-.5cm}
\end{split}
\end{equation}
Noting that $2(1-p_b-p_t)^{m_0}\leq 2e^{-m_0(p_b+p_t)}= 2e^{-k_0}<1$ for all $k_0\geq 1$, all the summands above are positive. Therefore, we conclude that $du/dp_t>0$, which implies $dP_Y^m(m_0)/d p_t<0$.
}
\subsection{Proof of Theorem~\ref{thm:noncvx}}\label{app:noncvx}

Due to Assumption~\ref{ass:noncvxball}, the iterates generated by the algorithm stay in ${\cal X}$ without projection. Hence, we proceed to analyze the convergence of the algorithm without projection.

Since $f(\cdot,z)$ is $L$-smooth, $F(\cdot)$ is $L$-smooth:
\begin{align}
     &F(x_{t+1}){-}F(x_t){\leq}\langle\nabla F(x_t),x_{t+1}{-}x_t\rangle{+}\frac{L}{2}\|x_{t+1}{-}x_t\|^2 \nonumber \\
     &\leq-\gamma\langle\nabla F(x_t),\frac{\hat{\hat{g}}_t}{\|\hat{\hat{g}}_t\|}\rangle+\frac{L\gamma^2}{2}\\
     \begin{split}
         &=-\gamma(1-Z_t)\langle\nabla F(x_t),\frac{\hat{\hat{g}}^{Z_t=0}_t}{\|\hat{\hat{g}}^{Z_t=0}_t\|}\rangle\\
         &\hspace{1cm}-\gamma Z_t\langle\nabla F(x_t),\frac{\hat{\hat{g}}^{Z_t=1}_t}{\|\hat{\hat{g}}^{Z_t=1}_t\|}\rangle+\frac{L\gamma^2}{2}
     \end{split}\\
     \begin{split}
          &=-\gamma(1-Z_t)\langle\nabla F(x_t),\frac{\nabla F(x_t)}{\|\nabla F(x_t)\|}\rangle\\
          &+\gamma(1-Z_t)\langle\nabla F(x_t),\frac{\nabla F(x_t)}{\|\nabla F(x_t)\|}-\frac{\hat{\hat{g}}^{Z_t=0}_t}{\|\hat{\hat{g}}^{Z_t=0}_t\|}\rangle\\
          &-\gamma Z_t\langle\nabla F(x_t),\frac{\hat{\hat{g}}^{Z_t=1}_t}{\|\hat{\hat{g}}^{Z_t=1}_t\|}\rangle+\frac{L\gamma^2}{2}
     \end{split}\\
     &\leq {-}\gamma (1{-}2Z_t)\|\nabla F(x_t)\|{+}\gamma (1{-}Z_t)\|\nabla F(x_t)\|\|e_t\|{+}\frac{L\gamma^2}{2},\label{eq:thmnoncvxcheckpt1}
\end{align}
where 
\begin{equation}
    e_t=\frac{\nabla F(x_t)}{\|\nabla F(x_t)\|}-\frac{\hat{\hat{g}}^{Z_t=0}_t}{\|\hat{\hat{g}}^{Z_t=0}_t\|}.
\end{equation}
Similar to Proof of Theorem~\ref{thm:stronglycvx} we define the following sets:
\begin{itemize}[wide, labelindent=0pt,topsep=.5mm]
    \item Define ${\cal T}^t$ as the set of $(1-\alpha_2)N$ agents with smallest indices $i\in[N]$ for which $Y_{i,t}=0$, i.e., 
    \begin{equation}
        {\cal T}^t=\{i|i\in[N], Y_{i,t}=0, \sum_{i}1=(1-\alpha_2)N\}
    \end{equation}
    such that $\sum_{i\in {\cal T}^t}i$ is minimized.
    \item For all agents $i\in{\cal T}^t$, define ${\cal T}^t_i$ as the set of $(1-\alpha_1)m$ smallest time indices $\tau\in[t-m+1,t]$ for which $W_{i,\tau}=0$, i.e.,
    \begin{equation}
        {\cal T}^t_i{=}\{\tau|\tau\in[t{-}m{+}1,t], W_{i,\tau}{=}0, \sum_\tau 1{=}(1-\alpha_1)m\}
    \end{equation}
    such that ${\sum_{\tau\in{\cal T}^t_i}}\tau$ is minimized.
\end{itemize}

We now use Lemma~\ref{lem:errorbound} to bound $\|e_t\|$  and take expectation of both sides with respect to $z\sim {\cal D}$ noting that $\nabla F_{i,\tau}(x_t)=\nabla F_{i,\tau'}(x_t)$, $\forall t,\tau,\tau'$:
\begin{align}
   \nonumber&\underset{z\sim{\cal D}}{\mathbb{E}}[F(x_{t+1}){-}F(x_t)]\leq {-}\gamma(1{-}2Z_t)\underset{z\sim{\cal D}}{\mathbb{E}}[\|\nabla F(x_t)\|]{+}\frac{L\gamma^2}{2}\\
    \nonumber&+2L\gamma^2 (1{-}Z_t)(m{-}1)(1{+}C_{\alpha_1}{+}2C_{\alpha_2}(C_{\alpha_1}{+}1))\\
    \nonumber&+2\gamma(1{-}Z_t)\underset{z\sim{\cal D}}{\mathbb{E}}\left[\left\|\nabla F(x_t){-}\frac{1}{(1{-}\alpha_2)N}\sum_{i\in{\cal T}^t}\nabla F_{i,t}(x_t)\right\|\right]\\
    &+4\gamma(1-Z_t)C_{\alpha_2}\underset{z\sim{\cal D}}{\mathbb{E}}[\underset{i\in{\cal T}^t}{\max}\|\nabla F_{i,t}(x_t)-\nabla F(x_t)\|_{\infty}]\\
    \begin{split}
        \label{eq:thmnoncvxcheckpt2}&\leq -\gamma(1-2Z_t)\underset{z\sim{\cal D}}{\mathbb{E}}[\|\nabla F(x_t)\|]+\frac{L\gamma^2}{2}\\
    &+2L\gamma^2 (m-1)(1+C_{\alpha_1}+2C_{\alpha_2}(C_{\alpha_1}+1))\\
    &+2\gamma\frac{\sigma}{\sqrt{(1-\alpha_2)Nb}}\\
    &+4\gamma C_{\alpha_2}\underset{\lambda\in(0,b/a)}{\inf} \left[\frac{\log{(2(1-\alpha_2)Nd)}+b\phi(\lambda/b)}{
    \lambda}\right]
    \end{split}\\
    \begin{split}
        &\leq -\gamma(1-2Z_t)\underset{z\sim{\cal D}}{\mathbb{E}}[\|\nabla F(x_t)\|]+\frac{L\gamma^2}{2}\\
    &+2L\gamma^2 (m-1)(1+C_{\alpha_1}+2C_{\alpha_2}(C_{\alpha_1}+1))\\
    &+2\gamma\frac{\sigma}{\sqrt{(1-\alpha_2)Nb}}+4\gamma C_{\alpha_2}\frac{a}{b}\log{(2(1-\alpha_2)Nd)}\\
    &+4\gamma C_{\alpha_2}\frac{\sigma\sqrt{2\log{((2(1-\alpha_2)Nd)}}}{\sqrt{b}},
    \end{split}
\end{align}
where the last two inequalities follow from \eqref{eq:35}-\eqref{eq:39} in proof of Theorem~\ref{thm:stronglycvx}.

Next step is to take expectation with respect to all randomness, where the challenge is to compute $\mathbb{E}[Z_t\|\nabla F(x_t)\|]$. Due to the same reasoning in proof of Theorem~\ref{thm:stronglycvx}, $x_t$ and $Z_t$ are dependent. random variables. Therefore, we use a similar trick and use total law of expectation by conditioning on the state at time $t-m+1-m_0$ for some $m_0\geq 0$, i.e.,
\begin{equation}\label{eq:trick}
    \mathbb{E}[Z_t\|\nabla F(x_t)\|]=\mathbb{E}[\mathbb{E}[Z_t\|\nabla F(x_t)\||{\cal S}_{t-m+1-m_0}]],
\end{equation}
where ${\cal S}_{t-m+1-m_0}=\{x_{t-m+1-m_0},\{\pi_{t-m+1-m_0}^i\}_{i\in[N]}\}$ and $\pi_{t-m+1-m_0}^i$ is the distribution of the state of agent $i$ at time $t-m+1-m_0$. Note that due to smoothness of $F(\cdot)$ and normalized updates:
\begin{equation}
    \|\nabla F(x_t)\|\leq \|\nabla F(x_{t-m+1-m_0})\|+L\gamma(m-1+m_0),
\end{equation}
and therefore \eqref{eq:trick} can be rewritten as
\begin{equation}
\begin{split}
        &\mathbb{E}[\mathbb{E}[Z_t\|\nabla F(x_t)\||{\cal S}_{t-m+1-m_0}]]\\
        &\leq
        \mathbb{E}[\|\nabla F(x_{t-m+m_0-1})\|\mathbb{E}[Z_t|{\cal S}_{t-m+1-m_0}]]\\
        &\hspace{.5cm}+\mathbb{E}[L\gamma(m-1+m_0)\mathbb{E}[Z_t|{\cal S}_{t-m+1-m_0}]]\label{eq:thmnoncvxcheckpt3}
\end{split}
\end{equation}


We now use Lemma~\ref{lem:practicalbound} (or Lemma~\ref{lem:markovianprobbound} in Appendix~\ref{app:tighterbound}) to establish uniform bounds on $\mathbb{E}[Z_t|{\cal S}_{t-m+1-m_0}]$:
\begin{align}
\begin{split}
        \mathbb{E}[\|\nabla F(x_t)\|Z_t]&\leq \mathbb{E}[\|\nabla F(x_{t-m+m_0-1})\|]P_Z^m(m_0)\\
        &\hspace{1cm}+L\gamma(m-1+m_0)P_Z^m(m_0)
\end{split}\\
&\hspace{-2cm}\leq P_Z^m(m_0)\left(\mathbb{E}[\|\nabla F(x_t)\|]+2L\gamma(m-1+m_0)\right),\label{eq:32}
\end{align}
where the last inequality follows from smoothness of $F(\cdot)$ and normalized updates. Now we take expectation of \eqref{eq:thmnoncvxcheckpt2} with respect to all randomness and use \eqref{eq:32}:
\begin{equation}
    \begin{split}
        &\mathbb{E}[F(x_{t+1}){-}F(x_t)]{\leq} -\gamma(1{-}2P_Z^m(m_0))\mathbb{E}[\|\nabla F(x_t)\|]{+}\frac{L\gamma^2}{2}\hspace{-1.5cm}\\
        &+4L\gamma^2 P_Z^m(m_0)(m-1+m_0)\\
    &+2L\gamma^2 (m-1)(1+C_{\alpha_1}+2C_{\alpha_2}(C_{\alpha_1}+1))\\
    &+2\gamma\frac{\sigma}{\sqrt{(1-\alpha_2)Nb}}+4\gamma C_{\alpha_2}\frac{a}{b}\log{(2(1-\alpha_2)Nd)}\\
    &+4\gamma C_{\alpha_2}\frac{\sigma\sqrt{2\log{((2(1-\alpha_2)Nd)}}}{\sqrt{b}}
    \end{split}
\end{equation}
Noting that $F(x_t)\geq F(x^\star)$ for all $t$, we rearrange the terms, sum from $t=m+m_0$ to $t=T+m+m_0-1$ to get the following result for all $m_0\in {\mathbb N}_0$ for which $P_Z^m(m_0)<1/2$:
\begin{equation}\label{eq:thmnoncvxcheckpt4}
\begin{split}
        \frac{1}{T}\sum_{t=m+m_0}^{T+m+m_0-1}\mathbb{E}&[\|\nabla F(x_t)\|]\leq\\
        &\frac{\mathbb{E}[F(x_{m+m_0})]-F(x^\star)}{\gamma T(1-2P_Z^m(m_0))}+\overline{C}(m_0)\gamma\hspace{-1cm}\\
        &+\frac{2\sigma}{(1-2P_Z^m(m_0))\sqrt{(1-\alpha_2)Nb}}\\
        &+\frac{4C_{\alpha_2}a\log{(2(1-\alpha_2)Nd)}}{(1-2P_Z^m(m_0))b}\\
        &+\frac{4C_{\alpha_2}\sigma\sqrt{2\log{((2(1-\alpha_2)Nd)}}}{(1-2P_Z^m(m_0))\sqrt{b}},\hspace{-1cm}
\end{split}
\end{equation}
where
\begin{equation}
\begin{split}
       &\overline{C}(m_0)=L\Big(0.5+4P_Z^m(m_0)(m-1+m_0)\\
       &+2(m-1)(1+C_{\alpha_1}+2C_{\alpha_2}(C_{\alpha_1}+1))\Big).
\end{split}
\end{equation}
Next, we upper bound $\mathbb{E}[F(x_{m+m_0})]$ using smoothness of $F$:
\begin{align}
\begin{split}
     &F(x_{m+m_0})\leq F(x_1)+\langle\nabla F(x_1),x_{m+m_0}-x_1\rangle\\
     &\hspace{2cm}+\frac{L}{2}\|x_{m+m_0}-x_1\|^2
\end{split}\\
    \begin{split}
        &\leq F(x_1){+}\|\nabla F(x_1)\|\gamma(m{+}m_0{-}1){+}\frac{L}{2}\gamma^2(m{+}m_0{-}1)^2.
    \end{split}
\end{align}
Finally, we set $\gamma=\gamma_0/\sqrt{T}$ and plug the above inequality into \eqref{eq:thmnoncvxcheckpt4} to get the final result:
\begin{equation}
\begin{split}
               &\frac{1}{T}\sum_{t=m+m_0}^{T+m-1+m_0}\mathbb{E}[\|\nabla F(x_t)\|]\\
               &\leq\frac{F(x_1)-F(x^\star)}{\sqrt{T}\gamma_0(1-2P_Z^m(m_0))}+\frac{\overline{C}(m_0)\gamma_0}{\sqrt{T}}\hspace{-1cm}\\
        &+\frac{\|\nabla F(x_1)\|(m-1+m_0)}{T(1-2P_Z^m(m_0))}+\frac{L\gamma_0(m-1+m_0)^2}{2T^{3/2}(1-2P_Z^m(m_0))}\\
        &+\frac{2\sigma}{(1-2P_Z^m(m_0))\sqrt{(1-\alpha_2)Nb}}\\
        &+\frac{4C_{\alpha_2}a\log{(2(1-\alpha_2)Nd)}}{(1-2P_Z^m(m_0))b}\\
        &+\frac{4C_{\alpha_2}\sigma\sqrt{2\log{((2(1-\alpha_2)Nd)}}}{(1-2P_Z^m(m_0))\sqrt{b}},\hspace{-1cm}
\end{split}
\end{equation}

\subsection{Proof of Theorem~\ref{thm:stronglycvxiid}}\label{app:stronglycvxiid}
The proof is identical to that of Theorem~\ref{thm:stronglycvx} until \eqref{eq:thmstrongcvxcheckpt1}. Next, we plug \eqref{eq:errorbound} into \eqref{eq:thmstrongcvxcheckpt1} and take expectation of both sides with respect to $z\sim {\cal D}$, this time noting that $\nabla F_{i,\tau}(x_t)$ and $\nabla F_{i,\tau'}(x_t)$ are iid random variables $\forall t,\tau,\tau'$ such that $\tau\neq\tau'$ in the SA setting. This results in:
\begin{align}
    \begin{split}
        &\underset{z\sim{\cal D}}{\mathbb{E}}[\|x_{t+1}-x^\star\|^2]\leq \underset{z\sim{\cal D}}{\mathbb{E}}[\|x_t-x^\star\|^2]+\gamma^2\\
        &-\frac{2\gamma}{\kappa}(1-Z_t(1+\kappa))\underset{z\sim{\cal D}}{\mathbb{E}}[\|x_t-x^\star\|]\\
        &+\frac{4\gamma^2L(m-1)(1+C_{\alpha_1}+2C_{\alpha_2}(C_{\alpha_1}+1))}{\mu}\\
        &+\frac{4\gamma\sigma}{\mu\sqrt{(1-\alpha_2)N(1-\alpha_1)mb}}\\
        &+\frac{8\gamma C_{\alpha_2}}{\mu}\underset{z\sim{\cal D}}{\mathbb{E}}[\underset{i\in{\cal T}^t}{\max}\|\nabla F_{i,t}(x_t)-\nabla F(x_t)\|_{\infty}]\\
        &{+}\frac{4\gamma C_{\alpha_1}}{\mu(1{-}\alpha_2)N}\hspace{-.1cm}\sum_{i\in{\cal T}^t}\hspace{-.1cm}2\hspace{-.1cm}\underset{z\sim{\cal D}}{\mathbb{E}}[\underset{\tau,\tau'\in{\cal T}^t_i}{\max}\|\nabla F_{i,\tau}(x_{\tau'}){-}\nabla F(x_{\tau'})\|_{\infty}]\hspace{-1cm}\\
        &{+}\frac{8\gamma C_{\alpha_1}C_{\alpha_2}}{\mu}2\hspace{-.1cm}\underset{z\sim{\cal D}}{\mathbb{E}}[\underset{i\in {\cal T}^t,\tau,\tau'\in{\cal T}_i^t}{\max}\|\nabla F_{i,\tau}(x_{\tau'}){-}\nabla F(x_{\tau'})\|_{\infty}].\hspace{-1cm}
    \end{split}
\end{align}
Next, we use Theorem~\ref{thm:expectedmaximum} on the last three terms above, noting that the first maximization is over $2(1-\alpha_2)Nd$ sub-gamma random variables, the second maximization is over $4(1-\alpha_1)md$ sub-gamma random variables, and the last maximization is over $4(1-\alpha_1)m(1-\alpha_2)Nd$ sub-gamma random variables. The rest of the proof is identical to that of Theorem~\ref{thm:stronglycvx}, resulting in:
\begin{equation}
    \begin{split}
        \mathbb{E}&[\|x_{T+m+m_0}-x^\star\|^2]\leq\\
        &\left(\|x_{1}-x^\star\|+\gamma(m+m_0-1)\right)^2(1-c_0(m_0)\gamma)^T\\
        &+\frac{4\sigma}{\mu\sqrt{(1-\alpha_2)N(1-\alpha_1)mb}c_0(m_0)}+\frac{\overline{C}(m_0)\gamma}{c_0(m_0)}\\
        &+\frac{8a C_{\alpha_2}\log{(2(1-\alpha_2)Nd)}}{\mu c_0(m_0) b}\\
        &+\frac{8 C_{\alpha_2}\sigma\sqrt{2\log{(2(1-\alpha_2)Nd)}}}{\mu c_0(m_0)\sqrt{b}}\\
        &+\frac{8a C_{\alpha_1}\log{(4(1-\alpha_1)md)}}{\mu c_0(m_0)b}\\
        &+\frac{8 C_{\alpha_1}\sigma\sqrt{2\log{(4(1-\alpha_1)md)}}}{\mu c_0(m_0)\sqrt{b}}\\
        &+\frac{16aC_{\alpha_1} C_{\alpha_2}\log{(4(1-\alpha_2)N(1-\alpha_1)md)}}{\mu c_0(m_0) b}\\
        &+\frac{16C_{\alpha_1} C_{\alpha_2}\sigma\sqrt{2\log{(4(1-\alpha_2)N(1-\alpha_1)md)}}}{\mu c_0(m_0)\sqrt{b}},
    \end{split}
\end{equation}

\subsection{Proof of Theorem~\ref{thm:noncvxiid}}\label{app:noncvxiid}
The proof is identical to that of Theorem~\ref{thm:noncvx} until \eqref{eq:thmnoncvxcheckpt1}. Next, we plug \eqref{eq:errorbound} into \eqref{eq:thmnoncvxcheckpt1} and take expectation of both sides with respect to $z\sim {\cal D}$, this time noting that $\nabla F_{i,\tau}(x_t)$ and $\nabla F_{i,\tau'}(x_t)$ are iid random variables $\forall t,\tau,\tau'$ such that $\tau\neq\tau'$ in the SA setting:

\begin{align}
   \begin{split}
        &\underset{z\sim{\cal D}}{\mathbb{E}}[F(x_{t+1})-F(x_t)]\\
        &\leq -\gamma(1-2Z_t)\underset{z\sim{\cal D}}{\mathbb{E}}[\|\nabla F(x_t)\|]+\frac{L\gamma^2}{2}\\
    &+2L\gamma^2 (m-1)(1+C_{\alpha_1}+2C_{\alpha_2}(C_{\alpha_1}+1))\\
    &+2\gamma\frac{\sigma}{\sqrt{(1-\alpha_2)N(1-\alpha_1)mb}}\\
    &+4\gamma C_{\alpha_2}\underset{z\sim{\cal D}}{\mathbb{E}}[\underset{i\in{\cal T}^t}{\max}\|\nabla F_{i,t}(x_t)-\nabla F(x_t)\|_{\infty}]\\
    &{+}\frac{2\gamma C_{\alpha_1}}{(1{-}\alpha_2)N}\hspace{-.1cm}\sum_{i\in{\cal T}^t}\hspace{-.1cm}2\hspace{-.1cm}\underset{z\sim{\cal D}}{\mathbb{E}}[\underset{\tau,\tau'\in{\cal T}^t_i}{\max}\|\nabla F_{i,\tau}(x_{\tau'}){-}\nabla F(x_{\tau'})\|_{\infty}]\hspace{-1cm}\\
        &{+}4\gamma C_{\alpha_1}C_{\alpha_2}2\hspace{-.1cm}\underset{z\sim{\cal D}}{\mathbb{E}}[\underset{i\in {\cal T}^t,\tau,\tau'\in{\cal T}_i^t}{\max}\|\nabla F_{i,\tau}(x_{\tau'}){-}\nabla F(x_{\tau'})\|_{\infty}].\hspace{-1cm}
    \end{split}
    \end{align}
Next, we use Theorem~\ref{thm:expectedmaximum} on the last three terms above, noting that the first maximization is over $2(1-\alpha_2)Nd$ sub-gamma random variables, the second maximization is over $4(1-\alpha_1)md$ sub-gamma random variables, and the last maximization is over $4(1-\alpha_1)m(1-\alpha_2)Nd$ sub-gamma random variables. The rest of the proof is identical to that of Theorem~\ref{thm:noncvx}, resulting in:
\begin{equation}
    \begin{split}
               &\frac{1}{T}\sum_{t=m+m_0}^{T+m-1+m_0}\mathbb{E}[\|\nabla F(x_t)\|]\\
               &\leq\frac{F(x_1)-F(x^\star)}{\sqrt{T}\gamma_0(1-2P_Z^m(m_0))}+\frac{\overline{C}(m_0)\gamma_0}{\sqrt{T}}\hspace{-1cm}\\
        &+\frac{\|\nabla F(x_1)\|(m-1+m_0)}{T(1-2P_Z^m(m_0))}+\frac{L\gamma_0(m-1+m_0)^2}{2T^{3/2}(1-2P_Z^m(m_0))}\hspace{-1cm}\\
        &+\frac{2\sigma}{(1-2P_Z^m(m_0))\sqrt{(1-\alpha_2)N(1-\alpha_1)mb}}\\
        &+\frac{4C_{\alpha_2}a\log{(2(1-\alpha_2)Nd)}}{(1-2P_Z^m(m_0))b}\\
        &+\frac{4C_{\alpha_2}\sigma\sqrt{2\log{((2(1-\alpha_2)Nd)}}}{(1-2P_Z^m(m_0))\sqrt{b}}\\
        &+\frac{4C_{\alpha_1}a\log{(4(1-\alpha_1)md)}}{(1-2P_Z^m(m_0))b}\\
        &+\frac{4C_{\alpha_1}\sigma\sqrt{2\log{(4(1-\alpha_1)md)}}}{(1-2P_Z^m(m_0))\sqrt{b}}\\
        &+\frac{8C_{\alpha_1}C_{\alpha_2}a\log{(4(1-\alpha_2)N(1-\alpha_1)md)}}{(1-2P_Z^m(m_0))b}\\
        &+\frac{8C_{\alpha_1}C_{\alpha_2}\sigma\sqrt{2\log{(4(1-\alpha_2)N(1-\alpha_1)md)}}}{(1-2P_Z^m(m_0))\sqrt{b}}
\end{split}
\end{equation}

\subsection{Tighter Bound on $P_Z(m_0)$}\label{app:tighterbound}

\begin{lemma}\label{lem:markovianprobbound}
Given the network and algorithm parameters $(m,N,\alpha_1,\alpha_2,M)$, the following holds for all $m_0\in{\mathbb{Z}}^+$:
\begin{equation}\label{eq:PZ}
P_Z(m_0)\leq\sum_{k=\alpha_2 N+1}^N \binom{N}{k}(P_Y^m(m_0))^k(1-P_Y^m(m_0))^{(1-k)},
\end{equation}
where
\begin{equation}
        P_Y^m(m_0)= \sum_{s=0}^1\sum_{k=\alpha_1m+1}^m r_s(k;m,1-p_t,p_b)\Pi_{m_0}(s)
\end{equation}
with
\begin{equation}
    \begin{split}
        &r_s(k;m,1-p_t,p_b)\\
        &=\hspace{-.5cm}\sum_{i=0}^{\min(k,n-k)}\binom{n-i}{k}\binom{k}{i}(1-p_t)^{k-i}(1-p_b)^{n-k-i}(1-p_t-p_b)^i\\
        &+\hspace{-.5cm}\sum_{i=0}^{\min(k-1+s,n-k-s)}\binom{n-i-1}{k-1+s}\binom{k-1+s}{i}(1-p_t)^{n-k-i-1+s}\\
        &\hspace{3cm}\times(1-p_b)^{n-k-i-s}(1-p_t-p_b)^{i+1}
    \end{split}
\end{equation}
and 
\begin{align}
    &\Pi_{m_0}(0)=\frac{p_t-p_t(1-p_b-p_t)^{m_0}}{p_b+p_t}\\
    &\Pi_{m_0}(1)=\frac{p_b+p_t(1-p_b-p_t)^{m_0}}{p_b+p_t}
\end{align}
\end{lemma}
\begin{proof}
The goal here is to find a tighter upper bound on $P_Z(m_0)$ than the bound provided in Lemma~\ref{lem:practicalbound}. The proof is identical to that of Lemma~\ref{lem:practicalbound} until \eqref{eq:PY}. A tighter bound is established by deriving the exact expression on $P_Y^m(m_0)$ rather than using a Chernoff's bound. We continue from \eqref{eq:PY}:
\begin{align}
      \nonumber P_Y^m(m_0){=}\underset{i\in[N]}{\max}\sum_{s=0}^1&\mathbb{P}\left( \frac{1}{m}\sum_{\tau=t-m+1}^t
       \hspace{-.4cm}W_{i,\tau}{>}\alpha_1|W_{i,t-m+1}=s\right)\\
       &\times\mathbb{P}_{\pi^i_{t-m+1-m_0}}(W_{i,t-m+1}{=}s),
\end{align}
where $\mathbb{P}_{\pi^i_{t-m+1-m_0}}(W_{i,t-m+1}=s)$ is the probability that $W_{i,t-m+1}=s$ given the distribution at time $t-m+1-m_0$. The first multiplicand in the above equation has a closed form as follows \cite{helgert1970sums}:
\begin{align}
    &\nonumber\mathbb{P}\left( \frac{1}{m}\sum_{\tau=t-m+1}^tW_{i,\tau}>\alpha_1|W_{i,t-m+1}=s\right)\\
    &=\sum_{k=\alpha_1m+1}^m\mathbb{P}\left( \sum_{\tau=t-m+1}^tW_{i,\tau}=k|W_{i,t-m+1}=s\right)\\
    &=\sum_{k=\alpha_1m+1}^m r_s(k;m,1-p_t,p_b),
\end{align}
where
\begin{equation}
    \begin{split}
        &r_s(k;m,1-p_t,p_b)\\
        &=\hspace{-.5cm}\sum_{i=0}^{\min(k,m-k)}\binom{m{-}i}{k}\binom{k}{i}(1{-}p_t)^{k-i}(1{-}p_b)^{m{-}k{-}i}(1{-}p_t{-}p_b)^i\\
        &+\hspace{-1cm}\sum_{i=0}^{\min(k-1+s,m-k-s)}\binom{m{-}i{-}1}{k{-}1{+}s}\binom{k{-}1{+}s}{i}(1{-}p_t)^{m{-}k{-}i{-}1{+}s}\\
        &\hspace{2cm}\times(1-p_b)^{m-k-i-s}(1-p_t-p_b)^{i+1}.
    \end{split}
\end{equation}
Next, we determine $\mathbb{P}_{\pi^i_{t-m+1-m_0}}(W_{i,t-m+1}=s)$ for $s=\{0,1\}$ as follows:
\begin{align}
        \nonumber&\mathbb{P}_{\pi_{t-m+1-m_0}^i}(W_{i,t-m+1}=0)=\frac{p_t}{p_b+p_t}\\
        \label{eq:initial0}&+\frac{(1{-}p_b{-}p_t)^{m_0}}{p_b+p_t}(\pi^i_{t{-}m{+}1{-}m_0}(0)p_b{-}\pi_{t{-}m{+}1{-}m_0}^i(1)p_t)\\
    \label{eq:initial1}&\mathbb{P}_{\pi^i_{t-m+1-m_0}}(W_{i,t-m+1}{=}1)=1{-}P_{\pi_{t-m+1-m_0}^i}(W_{i,t-m+1}{=}0)
\end{align}
The above probabilities depend on the distribution at time $t-m+1-m_0$. Noting that $\sum_{k=1}^mr_1(k;\cdot)>\sum_{k=1}^mr_0(k;\cdot)$ (as shown in Lemma 1), we upper bound \eqref{eq:PY} by upper bounding $\mathbb{P}_{\pi_{t-m+1-m_0}^i}(W_{i,t-m+1}=1)$. To do so, we lower bound \eqref{eq:initial0} by setting $\pi_{t-m+1-m_0}(1)=1$, i.e., by assuming that at time $t-m+1-m_0$ the agent was Byzantine, which is the worst-case intuition. All in all we have:
\begin{equation}
    P_Y^m(m_0)= \sum_{s=0}^1\sum_{k=\alpha_1m+1}^m r_s(k;m,1-p_t,p_b)\Pi_{m_0}(s),
\end{equation}
where
\begin{align}
    &\Pi_{m_0}(0)=\frac{p_t-p_t(1-p_b-p_t)^{m_0}}{p_b+p_t}\\
    &\Pi_{m_0}(1)=\frac{p_b+p_t(1-p_b-p_t)^{m_0}}{p_b+p_t}
\end{align}
\end{proof}

\end{document}